\documentclass[11pt]{article}

\usepackage[T1]{fontenc}
\usepackage{amssymb}
\usepackage{amsmath,amsfonts}
\usepackage{amsthm}
\usepackage{latexsym}
\usepackage{mathrsfs}
\usepackage{color}
\usepackage{comment}
\usepackage{algorithm}
\usepackage{algorithmic}
\usepackage{subfigure}
\usepackage{dsfont}
\usepackage[title]{appendix}
\usepackage{bbm}
\usepackage{multirow}
\usepackage{varioref} 
\usepackage[colorlinks=true,
            linkcolor=blue,
            anchorcolor=blue,
            citecolor=blue,
            urlcolor=blue,
            ]{hyperref}
\usepackage{cleveref} 
\usepackage{url}            
\usepackage{booktabs}       
\usepackage[a4paper]{geometry}
\usepackage{bm}
\usepackage{graphicx}
\usepackage[numbers]{natbib}
\usepackage{amsopn}

\newcommand{\email}[1]{\protect\href{mailto:#1}{#1}}

\newcommand*{\medcup}{\mathbin{\scalebox{1.5}{\ensuremath{\cup}}}}%

\newtheorem{lemma}{Lemma}
\newtheorem{thm}{Theorem}
\newtheorem{coro}{Corollary}
\newtheorem{defi}{Definition}

\newtheorem{assumption}{Assumption}

\newtheorem{remark}{Remark}
\newtheorem{exam}{Example}

\newcounter{spb}
\DeclareMathOperator{\dist}{dist}

\DeclareMathOperator{\dom}{dom}

\newcommand{\be}{\begin{equation}}
\newcommand{\ee}{\end{equation}}

\newcommand{\RNum}[1]{\uppercase\expandafter{\romannumeral #1\relax}}
\newcommand{\rNum}[1]{\lowercase\expandafter{\romannumeral #1\relax}}

\newcommand{\argmin}{\arg\min}

\newcommand{\st}{\mbox{s.t.}}

\def\R{{\mathbb R}}
\def\P{{\mathbb{P}}}
\def\1{\mathbbm{1}}

\def\mA{\mathcal{A}}

\def\mE{\mathcal{E}}
\def\mI{\mathcal{I}}

\def\mM{\mathcal{M}}
\def\mN{\mathcal{N}}
\def\mS{\mathcal{S}}
\def\mT{\mathcal{T}}

\def\mW{\mathcal{W}}
\def\mX{\mathcal{X}}

\def\mZ{\mathcal{Z}}
\def\mP{\mathcal{P}}
\def\bA{\bm{A}}

\def\b0{\bm{0}}
\def\ba{\bm{a}}
\def\bb{\bm{b}}

\def\bd{\bm{d}}
\def\be{\bm{e}}

\def\bmu{\bm{\mu}}
\def\bnu{\bm{\nu}}
\def\bo{\bm{1}}
\def\bs{\bm{s}}
\def\bu{\bm{u}}

\def\bw{\bm{w}}
\def\bx{\bm{x}}
\def\by{\bm{y}}
\def\bxi{\bm{\xi}}
\def\bz{\bm{z}}

\usepackage{amsfonts}

\date{\today}
\author{
Peng Wang\thanks{Department of Electrical Engineering and Computer Science, University of Michigan, Ann Arbor, USA. (\email{peng8wang@gmail.com}).}
\and Rujun Jiang\thanks{Corresponding author. School of Data Science, Fudan University, Shanghai, China. (\email{rjjiang@fudan.edu.cn}).}
\and Qingyuan Kong\thanks{ School of Data Science, Fudan University, Shanghai, China. (\email{qykong21@m.fudan.edu.cn}).}
\and Laura Balzano\thanks{Department of Electrical Engineering and Computer Science, University of Michigan, Ann Arbor, USA. (\email{girasole@umich.edu}.)}
}
\title{A Proximal DC Algorithm for Sample Average Approximation of Chance Constrained Programming\footnote{This paper has been accepted for publication in INFORMS Jounral on Computing and is available at \url{https://doi.org/10.1287/ijoc.2024.0648} \citep{wang2025proximal}. The copyright of the final published version is held by INFORMS.}}
\usepackage{varioref}

\begin{document}

\maketitle
\begin{abstract}
Chance constrained programming (CCP) refers to a type of optimization problem with uncertain constraints that are satisfied with at least a prescribed probability level. 
    In this work, we study the sample average approximation (SAA) method for chance constraints, which is an important approach to CCP in the data-driven setting where only a sample of multiple realizations of the random vector in the constraints is available. The SAA method approximates the underlying distribution with an empirical distribution over the available sample. 
    Assuming that the functions in the chance constraints are all convex, we reformulate the SAA of chance constraints into a difference-of-convex (DC) form. Additionally, by assuming the objective function is also a DC function, we obtain a DC constrained DC program. 
    To solve this reformulation, we propose a proximal DC algorithm and show that the subproblems of the algorithm are suitable for off-the-shelf solvers in some scenarios. Moreover, we not only prove the subsequential and sequential convergence of the proposed algorithm but also derive the iteration complexity for finding an approximate Karush-Kuhn-Tucker point. To support and complement our theoretical development, we show via numerical experiments that our proposed approach is competitive with a host of existing approaches. Our publicly available code repository can be found at \url{https://github.com/peng8wang/2024.0648}. 
\end{abstract}


\section{Introduction}

Chance constrained programming is a powerful modeling paradigm for optimization problems with uncertain constraints, which has found wide applications in diverse fields, such as finance \citep{bonami2009exact, ghaoui2003worst}, power systems \citep{bienstock2014chance, xie2017distributionally}, and supply chain \citep{chen2001inventory, gurvich2010staffing}, to name a few; see, e.g., \cite{kuccukyavuz2022chance} and the references therein for more applications. In general, a chance constrained program is to minimize a targeted loss subject to the probability of violating uncertain constraints being within a prespecified risk level. In this work, we consider a chance constrained program of the form
\begin{align}\label{P:CCP}
\min_{\bx \in \mX}\ \left\{f(\bx):\ \P \left(c_i(\bx,\bxi)\le 0, i \in \{1,\dots,m \}\right) \ge 1-\alpha \right\},
\end{align}
where the vector $\bx\in \R^n$ denotes the decision variables, the set $\mX$ is a deterministic set contained in the open set $U \subseteq \R^n$, $\bxi \in \R^d$ is a random vector with its probability distribution supported on a set $\Xi \subseteq \R^d$, $f:U \rightarrow \R$ and $c_i:U \times \Xi \rightarrow\R$ for all $i \in \{1,\dots,m\}$ are real-valued functions, and $\alpha\in(0,1)$ is a given risk parameter. This problem is known as a \emph{single chance constrained program} if $m=1$, and a \emph{joint chance constrained program} otherwise.

Problem \eqref{P:CCP} 
is generally difficult to optimize due to the following fundamental challenges. First, the feasible region formed by the chance constraint may be non-convex even if $c_i(\bx,\bxi)$ for each $i\in \{1,\dots,m\}$ is linear in $\bx$ and $\mX$ is a polyhedron \citep{luedtke2010integer}. 
Moreover, in the setting where a sample of $N$ i.i.d. realizations $\{\hat{\bxi}^i\}_{i=1}^N$  of the random vector $\bxi$ is available, while its distribution is unknown, it is generally impossible to compute the probability of satisfying the constraint for a given $\bx \in \mX$. To approximately solve Problem \eqref{P:CCP}, we consider its sample average approximation (SAA) over the sample $\{\hat{\bxi}^i\}_{i=1}^N$, which has been studied in \citep{ahmed2008solving, luedtke2008sample, pagnoncelli2009sample, pena2020solving}, as follows:
\begin{align}\label{P:CCP sample}
\min_{\bx \in \mX}\ \left\{ f(\bx):\ \frac{1}{N}\sum_{i=1}^N\1{\{C(\bx,\hat{\bxi}^i) \le 0\}} \ge 1 - \alpha  \right\},
\end{align}
where $C(\bx,\bxi) := \max\left\{ c_i(\bx,\bxi):i=1,\dots,m\right\}$ and $\1\{\cdot\}$ denotes the characteristic function, that is, $\1\{C(\bx,\hat{\bxi}^i) \le 0\} = 1$ if $C(\bx,\hat{\bxi}^i) \le 0$ and $0$ otherwise. In particular, it has been shown in \cite{luedtke2008sample,pagnoncelli2009sample} that solving Problem \eqref{P:CCP sample} can return a good approximate solution of Problem \eqref{P:CCP} when $N$ is sufficiently large. Moreover, Problem \eqref{P:CCP sample} is exactly equivalent to Problem \eqref{P:CCP} when the distribution of $\bm \xi$ is finite and discrete, with each event appearing with probability $1/N$. Although Problem \eqref{P:CCP sample} is deterministic and does not involve random variables, it remains challenging to optimize due to the discreteness of the constraint. 

Throughout this paper, we make the following assumptions on Problem \eqref{P:CCP sample}: 
\begin{assumption}\label{AS:1}
$\emph{(a)}$ The function $f$ takes the form of $f=g-h$, where $g$ and $h$ are continuous functions defined on an open set $U$ that contains $\mX$. Moreover, $h$ is convex, i.e.,  $h\left(\alpha \bm x + (1-\alpha)\bm y\right) \le \alpha h(\bm x) + (1-\alpha)h(\bm y)$ for all $\alpha \in [0,1]$, and $g$ is $\rho$-strongly convex for some $\rho \ge 0$, i.e., $g(\bx)-\rho\|\bx\|^2/2$ is convex. \\
$\emph{(b)}$ The set $\mX$ is non-empty, closed, and convex. \\
$\emph{(c)}$ The functions $c_i(\bx,\bxi)$ for all $i=1,\dots,m$ are convex and continuously differentiable  in $\bx$ on $\R^n$ for every $\bxi \in \Xi$.
\end{assumption}

In this paper, we study how to utilize these particular functional structures to develop an effective algorithmic framework for solving Problem \eqref{P:CCP sample}. 
Exploiting these structures, we reformulate Problem \eqref{P:CCP sample} into a DC constrained DC problem and propose a proximal DC algorithm for solving the reformulation. In the literature, existing approaches to solving Problem \eqref{P:CCP sample} generally can only prove subsequential convergence and lack iteration complexity analysis. In contrast to
these results, we not only prove the sequential convergence to a Karush-Kuhn-Tucker (KKT) point of the proposed algorithm but also derive the iteration complexity for finding an approximate KKT point. 

\subsection{Our Contributions}

In this work, we study the SAA \eqref{P:CCP sample} of the chance constrained program \eqref{P:CCP} when the distribution of the random vector $\bxi$ is unknown, but a sample of $N$ i.i.d. realizations $\{\hat{\bxi}^i\}_{i=1}^N$ of $\bxi$ is available. To solve this problem, we reformulate the SAA problem \eqref{P:CCP sample} into a DC constrained DC program by utilizing Assumption \ref{AS:1} and the empirical quantile function of $C(\bx,\bxi)$ over the sample $\{\hat{\bxi}^i\}_{i=1}^N$. Then, we propose a proximal DC algorithm (pDCA) for solving the reformulation, which proceeds by solving a sequence of convex subproblems by linearizing the second component of the obtained DC functions and adding a proximal term to the objective function.
In particular, we show that it is easy to compute the required subgradients by using the structure of the DC functions. Moreover, the obtained subproblem can be rewritten in a form that is suitable for off-the-shelf solvers. 
Finally, we analyze the convergence and iteration complexity of the proposed method. Specifically, we show that any accumulation point of the sequence generated by the proposed method is a KKT point of the reformulated problem under a constraint qualification. Next, we establish the sequential convergence along with its convergence rate using the Kurdyka-\L ojasiewicz (K\L) inequality with the associated exponent \cite{attouch2009convergence,attouch2010proximal,kurdyka1998gradients}. Moreover, we further show that the obtained DC program is equivalent to a convex constrained problem with a concave objective, which is amenable to the Frank-Wolfe (FW) method. By further showing the equivalence between proximal DC iterations for solving the DC program and modified FW iterations for solving the equivalent problem, we derive the iteration complexity of the pDCA for computing an approximate KKT point. In particular, in contrast to the standard iteration complexity of the FW method $O(1/\sqrt{k})$ (see, e.g., \cite{lacoste2016convergence}), the iteration complexity of our considered FW method is improved to $O(1/k)$ by utilizing the DC structure, where $k$ is the number of iterations. To support and complete our theoretical results, we conduct extensive experiments on both synthetic and real-world data sets. These experiments demonstrate the effectiveness of our proposed method. For implementation details and reproducibility, we refer the reader to \cite{wang2025proximal} and our publicly available code repository at \url{https://github.com/peng8wang/2024.0648}.

\subsection{Related Works}

We first review some popular methods for solving chance constrained programs and then briefly talk about some DC algorithms closely related to our work. Since the first appearance of chance constrained programs in \cite{charnes1959chance,charnes1958cost}, various algorithms for solving chance constrained problems under different settings have been proposed in a substantial body of literature over the past years. 
One well-known approach for solving Problem \eqref{P:CCP} is to reformulate the chance constraint into a convex constraint when the distribution of $\bxi$ is available. For example, \citet[Lemma 2.2]{henrion2007structural} showed that the chance constraint can be reformulated into a second-order cone if $C(\bx,\bxi)=\langle \bxi,\bx \rangle - b$, $\bxi$ has an elliptical symmetric distribution, and $b$ is a scalar.
We refer the reader to \cite{ghaoui2003worst,lagoa2005probabilistically,calafiore2006distributionally,henrion2008convexity,prekopa2003probabilistic} for more results on the convexity of the feasible region formed by chance constraints. These convex reformulations generally require a special distribution on random vector $\bxi$, such as Gaussian or log-concave distributions. 

However, in practice, sometimes only a few random sample points from the distribution of $\bxi$ are available while the distribution of $\bxi$ is unknown. To handle this scenario, one popular approach is to consider the SAA of the problem (see Problem \eqref{P:CCP sample}), which is obtained by replacing the true distribution with an empirical distribution corresponding to random sample points. \cite{luedtke2008sample} showed that the SAA with a risk level smaller than the required risk level can obtain a solution satisfying a chance constraint with high probability under suitable conditions. 
Later, \cite{pagnoncelli2009sample} showed that a solution of the SAA problem converges to that of the original problem with probability approaching one as $N$ goes to infinity. 
Despite the fact that it possesses nice convergence properties, the SAA problem \eqref{P:CCP sample} is generally difficult to optimize due to its discrete nature. To solve it, many different approaches have been proposed in the literature. For example, \cite{ahmed2008solving} proposed a mixed-integer programming (MIP) reformulation for the SAA problem; see also \cite{kuccukyavuz2022chance,luedtke2008sample,luedtke2010integer,ruszczynski2002probabilistic} and the references therein. \cite{curtis2018sequential} proposed a sequential algorithm, which minimizes quadratic subproblems with linear cardinality constraints iteratively.
\cite{bai2021augmented} proposed an augmented Lagrangian decomposition method for solving Problem \eqref{P:CCP} when $\bxi$ has a finite discrete distribution and $c_j(\cdot,\bxi)$ for $j=1,\dots,m$ are all affine.
Recently, \cite{pena2020solving} proposed a smoothing non-linear approximation of Problem \eqref{P:CCP sample} based on the empirical quantile of the chance constraint and developed a S$\ell_1$QP-type trust-region method to solve the approximation problem. Using a similar idea, \cite{shen2021sample} proposed a neural network model to approximate the empirical quantile of the chance constraint and employed a simulated annealing algorithm for solving the approximation problem. In general, some methods, such as \cite{shen2021sample}, are heuristic in nature, and some other works, such as \cite{bai2021augmented,curtis2018sequential,pena2020solving}, only establish subsequential convergence for their proposed methods and have no iteration complexity analysis. 
The scenario approximation approach proposed in \cite{calafiore2006scenario,nemirovski2006scenario} is another well-known sample-based approach for solving Problem \eqref{P:CCP}. This approach is simple and easy to implement, but it suffers from the solution becoming more and more conservative as the sample size increases. 

Another notable approach for solving Problem \eqref{P:CCP} is to consider its conservative and tractable approximations. 
Among these approximations, the most famous one is the condition value-at-risk (CVaR) approximation proposed by \cite{nemirovski2007convex}, which is based on a conservative and convex approximation of the indication function. In particular, \cite{hong2009simulating} proposed a gradient-based Monte Carlo method for solving the CVaR approximation. To avoid overly conservative solutions, \cite{hong2011sequential} studied a DC approximation of the chance constraint and tackled it by solving a sequence of convex approximations. 
\cite{xie2020bicriteria} proposed a bicriteria approximation for solving chance constrained covering problems and proved a constant factor approximation guarantee. More recently, \cite{jiang2022also} proposed a convex approximation named ALSO-X that always outperforms the CVaR approximation when uncertain constraints are convex. Moreover, \cite{kannan2021stochastic} proposed a stochastic approximation method for solving the chance-constrained nonlinear programs using smooth approximations. In addition, many other approximations have been studied for solving chance-constrained problems; see, e.g., \cite{shan2014smoothing,geletu2017inner,cao2020sigmoidal}.

Recently, \cite{laguel2024chance} applied bilevel optimization to solve chance constrained programs when the objective function and the constraints are convex with respect to the decision parameter. In addition, $p$-efficient point-based methods have been studied for solving chance constrained programs, where a $p$-efficient point is a realization of a random variable that lies within the top $p\%$ of all possible outcomes in terms of the value of the constraint function. \cite{dentcheva2000concavity} applied this method for solving chance constrained programs with discrete distributions. Later, \cite{kogan2014threshold,kogan2016erratum} extended this approach to solve joint chance constrained programs. \cite{cui2022nonconvex} considered two generalizations of chance constrained programs involving probabilities of disjunctive nonconvex functional events and mixed-signed affine combinations of the resulting probabilities. They proposed a new algorithmic approach that combines parameterized approximations, sampling-based expectation approximations, constraint penalization, and convexification to solve the generalized problems.


DC constrained DC programs refer to optimization problems that minimize a DC function subject to constraints defined by DC functions. 
 Such problems have been extensively studied in the literature for decades \citep{dinh2014recent,horst1999dc,le2018dc}. One of the most popular methods for solving DC programs is the DC algorithm and its variants, which solve a sequence of convex subproblems by linearizing the second component of DC functions \citep{hong2011sequential,lu2012sequential,tao1997convex}.
\cite{le2014dc} proposed a penalty method and a DC algorithm using slack variables and showed that every accumulation point of the generated sequence is a KKT point of the considered problem. Later, \cite{pang2017computing} studied the proximal linearized method for DC programs and showed that every accumulation point of the generated sequence is a Bouligand-stationary point under proper conditions.
Recently, \cite{van2021bundle} developed a proximal bundle method for addressing DC programs and analyzed its convergence under different settings. \cite{lu2022penalty} proposed penalty and augmented Lagrangian methods for solving DC programs, and established strong convergence guarantees for the proposed methods. 

\subsection{Notation and Definitions}

Besides the notation introduced earlier, we shall use the following notation throughout the paper. We write matrices in bold capital letters $\bA$, vectors in bold lower-case letters $\ba$, and scalars in plain letters $a$. Given a matrix $\bA \in \R^{m\times n}$, we use $a_{ij}$ to denote its $(i,j)$-th element. Given a vector $\bx \in \R^{n}$, we use $\|\bx\|$ to denote its Euclidean norm, $x_i$ its $i$-th element, and $x_{[M]}$ its $M$-th smallest element. We use $\bo$ and $\b0$ to denote the all-one vector and all-zero vector, respectively.

Next, we introduce some concepts in non-smooth analysis that will be used in our subsequent development. The details can be found in \cite{RW04}. Let $\varphi:\R^n \rightarrow (-\infty,\infty]$ be a given function. We say that the function $\varphi$ is \textit{proper} if $\mathrm{dom}(\varphi):=\{\bx\in \R^n: \varphi(\bx) < \infty\}\neq \emptyset$. A vector $\bs \in \R^n$ is said to be a \emph{Fr\'{e}chet subgradient} of $\varphi$ at $\bx \in \dom(\varphi)$ if
\begin{align}\label{eq:frech-subg}
\liminf_{\by\rightarrow\bx, \by\not=\bx} \frac{ \varphi(\by) - \varphi(\bx) - \langle \bs,\by-\bx \rangle }{ \|\by-\bx\|_2} \ge 0.
\end{align}
The set of vectors $\bs \in \R^n$ satisfying \eqref{eq:frech-subg} is called the \emph{Fr\'{e}chet subdifferential} of $f$ at $\bx \in \mathrm{dom}(\varphi)$ and denoted by $\widehat{\partial}\varphi(\bx)$. The \emph{limiting subdifferential}, or simply the \emph{subdifferential}, of $\varphi$ at $\bx\in \mathrm{dom}(\varphi)$ is defined as
\begin{align}
\partial \varphi(\bx) = \{ \bs \in\R^{n}: \exists \bx^k\rightarrow\bx, \bs^k \rightarrow \bm s \mbox{ with }\varphi(\bx^k)\rightarrow \varphi(\bx),\bs^k\in\widehat{\partial} \varphi(\bx^k)\}.
\notag 
\end{align}
When $\varphi$ is proper and convex,  thanks to \cite[Proposition 8.12]{RW04}, the limiting subdifferential of $\varphi$ at $\bx \in \mathrm{dom}(\varphi)$ coincides with the classic subdifferential defined as
\begin{align}\label{def:subg}
\partial \varphi(\bx) = \{\bs \in \R^n: \varphi(\by) \ge \varphi(\bx) + \langle \bs, \by-\bx \rangle,\ 
&\text{for all}\ \by \in \R^n\}.
\end{align}
We define the convex conjugate of a proper closed and convex function $\varphi$ as 
\begin{align*}
    \varphi^*(\bm y) = \sup_{\bm x \in \R^n} \left\{ \langle \bm y, \bm x \rangle - \varphi(\bm x) \right\}. 
\end{align*}

For a non-empty set $\mS \subseteq \R^n$, its {\em indicator function} $\delta_\mS: \R^n \rightarrow \{0,+\infty\}$ is defined as
\begin{align*}
\delta_\mS(\bx) = \begin{cases}
0, & \text{if}\ \bx \in \mS, \\
+\infty, & \text{otherwise}.
\end{cases}
\end{align*}
Its {\em normal cone} (resp. Fr\'echet normal cone) at $\bx \in \mS$ is defined as $\mN_{\mS}(\bx) := \partial \delta_{\mS}(\bx)$ (reps. $\widehat{\mN}_{\mS}(\bx) := \widehat{\partial} \delta_{\mS}(\bx)$). Moreover, its tangent cone at $\bm x \in \mathcal{S}$ is $\mathcal{T}_{\mathcal{S}}(\bm x) := \{\bm w \in \R^n:  (\bm x^k - \bm x)/\tau^k \to \bm w\ \text{for some}\ \bm x^{k} \to \bm x\ \text{with}\ \bm x^k \in \mathcal{S}\ \text{and}\ \tau^k \searrow 0  \}$. Given a point $\bx \in \R^n$, its distance to $\mS$ is defined as $\mathrm{dist}(\bx,\mS) = \inf_{\by \in \mS}\|\bx-\by\|$. We say that $\mS$ is \emph{regular} at one of its points $\bx$ if it is locally closed and satisfies $\mN_\mS(\bx)=\widehat{\mN}_{\mS}(\bx)$. In addition, we say that a function $\varphi$ is \emph{regular} at $\bx$ if $\varphi(\bx)$ is finite and its epigraph $\mathrm{epi}(\varphi)$ is regular at $(\bx,\varphi(\bx))$. Suppose that $\varphi$ is a convex function. The directional derivative of $\varphi$ at $\bx \in \R^n$ in the direction $\bd \in \R^n$ is defined by
\begin{align*}
\varphi^\prime(\bx,\bd) = \lim_{t\searrow 0}\frac{\varphi(\bx+t\bd)-\varphi(\bx)}{t}.
\end{align*}
In particular, it holds that
\begin{align}\label{eq:dir der}
\varphi^\prime(\bx,\bd) = \sup\left\{\langle \bs,\bd \rangle:\ \bs \in \partial \varphi(\bx) \right\}.
\end{align}

We next introduce the K\L\ property with the associated exponent; see, e.g., \cite{attouch2009convergence,attouch2010proximal,attouch2013convergence,kurdyka1998gradients}.
\begin{defi}[K\L\ property and exponent]\label{def:KL}
Suppose that $\varphi:\R^n\rightarrow (-\infty,\infty]$ is proper and lower semicontinuous. The function $\varphi$ is said to satisfy the K\L\ property at $\bar{\bx} \in \left\{ \bx \in \R^n: \partial \varphi(\bx) \neq \emptyset\right\}$
if there exist a constant $\eta \in (0,\infty]$, a neighborhood $U$ of $\bar{\bx}$, and a continuous concave function $\psi:[0,\eta)\rightarrow \R_+$ with $\psi(0)=0$, $\psi$ being continuously differentiable on $(0,\eta)$, and $\psi^\prime(s) > 0$ for $s \in (0,\eta)$ such that
\begin{align}\label{eq:KL ineq}
\psi^\prime\left( \varphi(\bx) - \varphi(\bar{\bx})\right) \mathrm{dist}(0,\partial \varphi(\bx)) \ge 1
\end{align}
for all $\bx \in U$ satisfying $\varphi(\bar{\bx}) < \varphi(\bx) < \varphi(\bar{\bx}) + \eta$. In particular, if $\psi(s)=cs^{1-\theta}$ for some $c>0$ and $\theta \in (0,1)$, then $\varphi$ is said to satisfy the K\L\ property at $\bar{\bx}$ with exponent $\theta$.
\end{defi}
\noindent It is worth mentioning that a wide range of functions arising in applications satisfy the K\L\ property, such as proper and lower semicontinuous semialgebraic functions \citep{attouch2010proximal}.

The rest of this paper is organized as follows. In \Cref{sec:algo}, we reformulate Problem \eqref{P:CCP sample} into a DC constrained DC program and introduce the proposed algorithm pDCA. In \Cref{sec:conv}, we analyze the convergence and iteration complexity of the proposed method. In \Cref{sec:exte}, we discuss some extensions of our approach. In \Cref{sec:expe}, we report the experimental results of the proposed method and other existing methods. We end the paper with some conclusions in \Cref{sec:conc}.	  

\section{A Proximal DC Algorithm for Chance Constrained Programs}\label{sec:algo}

In this section, we first reformulate Problem \eqref{P:CCP sample} 
into a DC constrained DC program based on the empirical quantile. Then, we propose a proximal DC algorithm (pDCA) for solving the reformulation. 
To proceed, we introduce some further notions that will be used in the sequel. Let
\begin{align}\label{def:c}
C(\bx,\bxi) := \max\left\{ c_i(\bx,\bxi):i=1,\dots,m\right\}.
\end{align}
Given a sample $\{\hat{\bxi}^i\}_{i=1}^N$, let
\begin{align}\label{def:cN}
\widehat{C}(\bx) := \left(C(\bx,\hat{\bxi}^1),\dots, C(\bx,\hat{\bxi}^N)\right) \in \R^N.
\end{align}
We define the $p$-th empirical quantile of $C(\bx,\bxi)$ over the sample $\{\hat{\bxi}^i\}_{i=1}^N$ for a probability $p \in (0,1)$ by
\begin{align*}
\hat{Q}_C(p) := \inf\left\{y \in \R: \frac{1}{N} \sum_{i=1}^N \1 {\{C(\bx,\hat{\bxi}^i) \le y\}} \ge p \right\}. 
\end{align*}
Throughout this section, let
\begin{align}\label{eq:M T}
M :=\lceil (1-\alpha)N \rceil.
\end{align} 

\subsection{DC Reformulation of the Chance Constraint}

In this subsection, we reformulate the sample-based chance constraint in Problem \eqref{P:CCP sample} into a DC constraint using the empirical quantile function of $C(\bx,\bxi)$ over the sample $\{\hat{\bxi}^i\}_{i=1}^N$. To begin, according to \cite[Chapter 21.2]{van2000asymptotic}, the $(1-\alpha)$-th empirical quantile of $C(\bx,\bxi)$ over the sample $\{\hat{\bxi}^i\}_{i=1}^N$ for $\alpha \in (0,1)$ is
\begin{align*}
\hat{Q}_C(1-\alpha) = \widehat{C}_{[M]}(\bx),
\end{align*}
where $\widehat{C}_{[M]}(\bx)$ denotes the $M$-th smallest element of $\widehat{C}(\bx)$. This leads to an equivalent reformulation of Problem \eqref{P:CCP sample} as follows:
\begin{align}\label{P:CCP quan}
\min_{\bx \in \mX}\left\{ f(\bx):\ \widehat{C}_{[M]}(\bx) \le 0\right\}.
\end{align}
We should mention that the empirical quantile constraint has been considered in the literature. For example, \cite{pena2020solving} considered smooth approximations of the quantile constraint, and \cite{cui2018portfolio} split the quantile constraint into some easier pieces by introducing new variables. 
In contrast, we directly handle the quantile constraint by reformulating it into a DC form.
To simplify our development, we denote the constraint set of Problem \eqref{P:CCP quan} by
\begin{align}\label{set:Z}
\mZ_M :=  \left\{\bx\in \R^n: \widehat{C}_{[M]}(\bx) \le 0\right\}.
\end{align}
Note that if $\alpha < 1/N$ and $M=N$, this constraint implies $C(\bx,\hat{\bxi}^i) \le 0$ for all $i\in [N]$. This, together with \Cref{AS:1}(c) and \eqref{def:c}, implies that the constraint set $\mZ_N$ is convex. For this case, Problem \eqref{P:CCP quan} minimizes a DC objective function subject to convex constraints, and many existing algorithms in the literature have been proposed to solve this problem; see, e.g., \cite{dinh2014recent} and the references therein. To avoid this case, we assume that $M \le N-1$ throughout this paper. Using the structure of the function $\widehat{C}(\cdot)$ and the convexity of $c_i(\cdot,\bxi)$, we show that the above constraint is equivalent to a DC constraint.

\begin{lemma}\label{lem:set-Z}
Suppose that Assumption \ref{AS:1} holds and that $M \le N-1$. Let
\begin{align}\label{eq:G H}
G(\bx) := \sum_{i=M}^N \widehat{C}_{[i]}(\bx),\  H(\bx) :=  \sum_{i=M+1}^N \widehat{C}_{[i]}(\bx).
\end{align}
Then, $G$ and $H$ are both continuous and convex functions, and the chance constraint in \eqref{set:Z} is equivalent to a DC constraint
\begin{align}\label{decom cN}
  G(\bx) - H(\bx) \le 0.
\end{align}
\end{lemma}
\begin{proof}

The continuity of $G$ and $H$ directly follows from \Cref{AS:1}(c), \eqref{def:c}, \eqref{def:cN}, and \eqref{eq:G H}. Since $H(\bx)$ denotes the sum of $N-M$ largest components of $\widehat{C}(\bx)$, we rewrite it as
\begin{align}\label{eq3:lem-set-Z} 
H(\bx) = \max\left\{\sum_{t=1}^{N-M} \widehat{C}_{i_t}(\bx):\ {1 \le i_1 < i_2 < \dots < i_{N-M} \le N} \right\}.
\end{align}
According to the convexity of $c_j(\bx,\hat{\bxi}^i)$ for all $i=1,\dots,N$ and $j=1,\dots,m$ due to Assumption \ref{AS:1}(c) and the fact that the pointwise maximum of convex functions is still convex \cite[Proposition 2.1.2]{hiriart2004fundamentals}, we see that $C(\bx,\hat{\bxi}^i)$ for all $i=1,\dots,N$ are convex.
This, together with \eqref{eq3:lem-set-Z}, the fact that the sum of convex functions is convex, and the fact that the pointwise maximum of convex functions is still convex, implies that $H(\bx)$ is convex.  By the same argument, we show that $G(\bx)$ is convex. Given $\bz \in \R^N$ and $M \le N-1$, one can decompose $z_{[M]}$ as
\begin{align}\label{eq:decomp zM}
z_{[M]} = \sum_{i=M}^N z_{[i]} -  \sum_{i=M+1}^N z_{[i]},\ & \text{for all}\ M=1,\dots,N-1.
\end{align}
This, together with \eqref{eq:G H}, implies that $\widehat{C}_{[M]}(\bx) \le 0$ is equivalent to \eqref{decom cN}. 

\end{proof} 

Consequently, using Lemma \ref{lem:set-Z} and Assumption \ref{AS:1}(a), Problem \eqref{P:CCP quan} can be cast as the following DC constrained DC program:
\begin{align}\label{CCP:sample-DC}
\begin{aligned}
\min_{\bx \in \mX}\ f(\bx) := g(\bx) - h(\bx)\qquad \st\ \ G(\bx) - H(\bx) \le  0,
\end{aligned}
\end{align}
where $g,h$ are continuous and convex and $G,H$ defined in \eqref{eq:G H} are also continuous and convex. 

\subsection{A Proximal DC Algorithm for Chance Constrained Programs}
In this subsection, we propose a proximal DC algorithm for solving Problem \eqref{CCP:sample-DC}. To begin, we define
\begin{align}
& \mI := \left\{(i_1,i_2,\dots,i_{N-M}):1 \le i_1 < i_2 < \dots < i_{N-M} \le N\right\},\label{set:I}
\end{align}
and denote the active index set of $C(\bx,\hat{\bxi}^i)$ and $H(\bx)$ in \eqref{eq:G H} respectively by
\begin{align}
& \mM_{c}^i(\bx) := \left\{ j\in\{1,\dots,m\}: c_j(\bx,\hat{\bxi}^i)=C(\bx,\hat{\bxi}^i)\right\}, \label{index:c i}\\
& \mM_H(\bx) := \left\{I \in \mI: \sum_{t=1}^{N-M} \widehat{C}_{i_t}(\bx) =  H(\bx)\right\}.\label{index:h H}
\end{align}
These equations define the set of indices $j$ where $c_j(\bx,\hat{\bxi}^i)$ attains the maximum (Eq \eqref{index:c i}) and the index set of the $N-M$ largest elements in $\{C(\bx,\hat{\bxi}^{i})\}_{i=1}^N$ (Eq \eqref{index:h H}), respectively. 
Next, we specify how to compute the subgradient of $H(\bx)$ efficiently by utilizing its structure. 
\begin{lemma}\label{lem:subg H}
Suppose that Assumption \ref{AS:1} holds. 
Let $H$ be defined in \eqref{eq:G H}. Given an $\bx \in \R^n$, it holds that
\begin{align}\label{subg:H}
\partial H(\bx) = {\rm conv}\left\{\medcup \sum_{t=1}^{N-M} \partial \widehat{C}_{i_t}(\bx): (i_1,\dots,i_{N-M}) \in \mM_H(\bx)\right\},
\end{align}
where
\begin{align}\label{subg:c i}
\partial \widehat{C}_{i}(\bx) = {\rm conv}\left\{\cup \{\nabla c_j(\bx,\hat{\bxi}^i)\}: j \in \mM_c^i(\bx)\right\}
\end{align}
for all $i=1,\dots,N$ and ${\rm conv}(\mathcal{A})$ denotes the convex hull of the set $\mathcal{A}$.
\end{lemma}
\begin{proof}

It follows from \eqref{eq3:lem-set-Z} and the rule of calculating the subdifferential of
the pointwise maximum of convex functions (see \Cref{lem:subdiff rule} (i)) that
\begin{align*}
\partial H(\bx) & = {\rm conv}\left\{\medcup\partial \sum_{t=1}^{N-M} \widehat{C}_{i_t}(\bx): (i_1,\dots,i_{N-M}) \in \mM_H(\bx)\right\} \notag \\
& = {\rm conv}\left\{\medcup \sum_{t=1}^{N-M} \partial \widehat{C}_{i_t}(\bx): (i_1,\dots,i_{N-M}) \in \mM_H(\bx)\right\},
\end{align*}
where the second equality follows from \Cref{lem:subdiff rule} (ii) and the relative interior of $\mathrm{dom}(c_j(\cdot,\bm \xi))$ is $\R^n$ due to the continuous differentiability of $c_j(\cdot,\bm \xi)$ on $\R^n$ by \Cref{AS:1}(c). Since $\widehat{C}_{i}(\bx) = C(\bx,\hat{\bxi}^i)=\max \{c_j(\bx,\hat{\bxi}^i):j=1,\dots,m \}$ for any $i\in\{1,\dots,N\}$, using the rule of calculating the subdifferential of the pointwise maximum of convex functions again and Assumption \ref{AS:1}(c), we obtain \eqref{subg:c i}. 

\end{proof}

Armed with the above setup, we are ready to propose a proximal DC algorithm for solving Problem \eqref{CCP:sample-DC}.
Specifically, suppose that an initial point $\bx^0 \in \mX$ satisfying $G(\bx^0) - H(\bx^0) \le 0$ is available. At the $k$-th iteration, we choose
$\bs_h^k \in \partial h(\bx^k)$ and $\bs_H^k \in \partial H(\bx^k)$, and generate the next iterate $\bx^{k+1}$ by solving the following convex subproblem
\begin{align}\label{DC:subproblem}
\begin{aligned}
\bx^{k+1} \in \argmin_{\bx \in \mX}\quad & g(\bx) -  h(\bx^k) - \langle \bs_h^k, \bx-\bx^k \rangle  + \frac{\beta}{2}\|\bx-\bx^k\|^2 \\
\st\quad &  G(\bx) - H(\bx^k) - \langle \bs_H^k, \bx-\bx^k \rangle  \le 0,
\end{aligned}
\end{align}
where $\beta \ge 0$ is a penalty parameter. As shown in Lemma \ref{lem:subg H}, the subgradient $\bs_H^k$ can be easily computed. However, Problem \eqref{DC:subproblem} is still not suitable for off-the-shelf solvers, because it is difficult to directly input $G(\bx)$ defined in \eqref{eq:G H}, which involves the sum of the $N-M+1$ largest components of $\widehat{C}(\bx,\bxi)$, into solvers due to its combinatorial nature. To address this issue, we reformulate Problem \eqref{DC:subproblem} into a form that is suitable for solvers by introducing an auxiliary variable $\bz \in \R^N$ such that $C(\bx,\hat{\bxi}^i) \le z_i$, for all $i=1,\dots,N$. Note that
\begin{align*}
\sum_{i=M}^Nz_{[i]} = \max_{\bu \in \R^N}\left\{\langle \bu,\bz\rangle:\ \bm{0} \le \bu \le \bo,\ \bo^T\bu=N-M+1 \right\}.
\end{align*}
This is a linear program and its dual problem is
\begin{align*}
\min_{\bm{\lambda} \in \R^N, \mu \in \R} \left\{\langle \bo,\bm{\lambda}\rangle + (N-M+1)\mu:\ \bz - \bm{\lambda} - \mu\bo \le \bm{0},\ \bm{\lambda} \ge \bm{0} \right\}.
\end{align*}
Using the strong duality of linear programming, we rewrite Problem \eqref{DC:subproblem} as
\begin{align} \label{DC:subp z}
\begin{aligned}
\bx^{k+1} = \argmin_{\bx \in \mX,\bz \in \R^N,\bm{\lambda} \in \R^N, \mu \in \R}\quad & g(\bx) - h(\bx^k) - \langle \bs_h^k, \bx-\bx^k \rangle + \frac{\beta}{2}\|\bx-\bx^k\|^2 \\
\st\quad & \langle \bo,\bm{\lambda}\rangle + (N-M+1)\mu - H(\bx^k) - \langle \bs_H^k, \bx-\bx^k \rangle  \le 0, \ \\
& \bz - \bm{\lambda} - \mu\bo \le \bm{0},\ \bm{\lambda} \ge \bm{0},  \\
& c_j(\bx,\hat{\bxi}^i) - z_i \le 0,\ \forall\ i=,1\dots,N,\ j=1,\dots,m.
\end{aligned}
\end{align}
We remark that we can eliminate the auxiliary variable $\bz \in \R^N$ by combining $c_j(\bx,\hat{\bxi}^i) - z_i \le 0$ for $i=1,\dots, N,~j=1,\ldots,m$ and $\bz - \bm{\lambda} - \mu\bo \le \bm{0}$ together and obtain $c_j(\bx,\hat{\bxi}^i) - \lambda_i - \mu  \le 0$ for $i=1,\dots, N,~j=1,\ldots,m$. We now summarize the proposed proximal DC algorithm in Algorithm \ref{alg-1}.

\begin{algorithm}[!hbtp]
	\caption{A Proximal DC Algorithm for Chance Constrained Programs}
	\begin{algorithmic}[1]
		\STATE \textbf{Input}: data sample $\{\hat{\bxi}_i\}_{i=1}^N$, feasible point $\bx^0$, $\beta \ge 0$.
		\FOR{$k=0,1,\dots$}
		\STATE take any $\bs_h^k \in \partial h(\bx^k)$ and $\bs_H^k \in \partial H(\bx^k)$
		\STATE solve Problem \eqref{DC:subp z} to obtain an $\bx^{k+1}$
		\IF{a termination criterion is met}
		\STATE stop and return $\bx^{k+1}$
		\ENDIF
	\ENDFOR 
	\end{algorithmic}
	\label{alg-1}
\end{algorithm}

Before we proceed, let us make some remarks on Algorithm \ref{alg-1}. 
First, Algorithm \ref{alg-1} is closely related to sequential convex programming methods in \cite{lu2012sequential,yu2021convergence}. However, different from them, we exploit the structure of the DC function and reformulate the subproblem into a form that is suitable for off-the-shelf solvers. 
Second, our DC approach significantly differs from that in \cite{hong2011sequential}. Given finite realizations $\{\hat{\bm \xi}^{i}\}_{i=1}^N$, the following DC approximation for the chance constraint is used in \cite{hong2011sequential}: 
\begin{align*}
\inf_{\epsilon > 0}\ \frac{1}{\epsilon N}\sum_{i=1}^N \left( \max\{\epsilon + C(\bm x, \hat{\bm \xi}^i),0\} -  \max\{C(\bm x, \hat{\bm \xi}^i),0\} \right) \le \alpha. 
\end{align*}  
In this approximation, the hyperparameter $\epsilon$ needs to be carefully tuned to achieve good performance in practice. Specifically, if $\epsilon$ is small, the approximation is better but the subproblem becomes ill-conditioned and difficult to solve. Conversely, if $\epsilon$ is large, the subproblem becomes easier to solve but the approximation is poor. In addition, this DC approach is a conservative approximation of the original problem. Compared to the DC approach in \cite{hong2011sequential}, our DC approach has two key advantages: (i) our approach directly applies DC reformulation to the empirical quantile of the chance constraint without any approximation (see Lemma 1), thereby avoiding the suboptimality caused by the conservative approximation; (ii) our reformulation of the chance constraint does not involve the hyperparameter $\epsilon$, making it simpler to implement.
Third, a key issue in our implementation is how to choose a feasible initial point $\bx^0$. A common approach is to solve a convex approximation of Problem \eqref{P:CCP sample} such as CVaR \citep{nemirovski2007convex} to generate a feasible point. Another typical approach is to use the (exact) penalization method to compute a feasible point \citep{lu2022penalty,le2012exact}. 
Finally, the subproblem \eqref{DC:subp z} is easy to solve in some scenarios. Specifically, it is observed that the functions $c_j(\cdot,\bxi)$ for all $j=1,\dots,m$ in many practical applications take a linear form; see, e.g., \cite{luedtke2010integer,kuccukyavuz2012mixing}. Based on this observation, suppose that in \eqref{CCP:sample-DC} $\mX$ is a polyhedron and
\begin{align}\label{ex:LP}
g(\bx) = \ba_0^T\bx,\ c_j(\bx,\bxi) =  \ba_j^T\bx + \bb_j^T\bxi,\ \text{for all}\ j =1,\dots,m.
\end{align}
Then, substituting \eqref{ex:LP} into \eqref{DC:subp z} with $\beta=0$ (resp. $\beta > 0$) yields a linear (resp. quadratic) program with $(m+2)N+1$ linear constraints (without considering the linear constraints in $\mX$). We can solve it easily by inputting it into off-the-shelf linear (resp. quadratic) programming solvers, such as \textsf{MOSEK}, \textsf{Gurobi}, and \textsf{CPLEX}. In addition, suppose that in \eqref{CCP:sample-DC} $\mX$ is a polyhedron and
\begin{align}\label{ex:QP}
g(\bx) = \bx^T\bA\bx + \ba_0^T\bx,\ c_j(\bx,\bxi) =  \ba_j^T\bx + \bb_j^T\bxi,\ \text{for all}\ j =1,\dots,m,
\end{align}
where $\bA \in \R^{n\times n}$ is a symmetric matrix. The resulting subproblem \eqref{DC:subp z} is a quadratic program when $\beta \ge 0$. 

In addition, we have some remarks on the penalty parameter $\beta \ge 0$. First, the penalty parameter can be updated in an adaptive manner as long as it is non-increasing and non-negative. In our numerical experiments, we observe that an adaptive scheme may empirically accelerate the convergence of the pDCA. Second, there is a trade-off between the parameters $\rho$ and $\beta$, where $\rho$ is the coefficient of strong convexity in \Cref{AS:1}. According to \Cref{thm:subseq,thm:entire conv} in the next section, it is required that $\rho+2\beta > 0$ to guarantee subsequential and global convergence. Notably, on one hand, the assumption of strong convexity for $g$ is not required as long as $\beta > 0$. On the other hand, when $\beta=0$, \Cref{alg-1} reduces to the standard DCA method. Then, our convergence analysis applies to DCA for solving DC programming when $\rho > 0$.

\section{Convergence and Iteration Complexity Analysis}\label{sec:conv}

In this section, we study the convergence properties of Algorithm \ref{alg-1}. Towards this end, we first show the subsequential convergence of the sequence $\{\bx^k\}$ generated by Algorithm \ref{alg-1} to a KKT point of Problem \eqref{CCP:sample-DC} under a constraint qualification. Second, we prove convergence of the entire sequence $\{\bx^k\}$ if in addition the K\L\ property holds for a tailor-designed potential function. Finally, we analyze the iteration complexity of Algorithm \ref{alg-1}. We point out that the proposed algorithm and its convergence apply to Problem \eqref{CCP:sample-DC} with $G(\cdot)$ and $H(\cdot)$ being general convex functions defined on an open set that contains $\mathcal X$, which takes the form of general DC constrained DC programs. An extension to multiple DC constraints will be discussed in \Cref{subsec:ext2}.

Before we proceed, we introduce some further notation, assumptions, and definitions that will be used throughout this section. To begin, we specify the convex constraints in the set $\mX$ as follows:
\begin{align}\label{set:X}
\mX = \left\{\bx \in \R^n: \ba_i^T\bx + b_i = 0,\ i \in \mE,\ \omega_i(\bx) \le 0,\ i \in \mI \right\},
\end{align}
where $\ba_i \in \R^n$ and $b_i \in \R$ for all $i \in \mE$, $\omega_i:\R^n \rightarrow \R$ for all $i \in \mI$ are convex and continuously differentiable functions, and $\mE$ and $\mI$ are finite sets of indices. We denote the active set of the inequality constraints at $\bx \in \mX$ by
  \begin{align}\label{set:active}
\mathcal{A}(\bx) := \left\{ i \in \mI: \omega_i(\bx) = 0 \right\},
\end{align}
 and the feasible set of Problem \eqref{CCP:sample-DC} by
\begin{align*}
\bar{\mX} := \left\{\bx \in \mX: G(\bx) - H(\bx) \le 0 \right\}.
\end{align*}
We now introduce a generalized version of the Mangasarian-Fromovitz constraint qualification (MFCQ), which is a widely used assumption on the algebraic description of the feasible set of constrained problems that ensures that the KKT conditions hold at any local minimum \citep{lu2012sequential,ye1995optimality}.

\begin{assumption}[Generalized MFCQ]\label{AS:MFCQ}
The generalized MFCQ of Problem \eqref{CCP:sample-DC} holds for every $\bm x \in \bar{\mathcal{X}}$, i.e., there exists $\bm y \in \mathcal{X}$ such that 
\begin{align}
\langle \nabla \omega_i(\bm x), \bm y - \bm x \rangle < 0,\ \text{for all}\ i \in \mathcal{A}(\bm x),\label{slater:1}
\end{align}
and if $G(\bm x) = H(\bm x)$, it holds that 
\begin{align} \label{slater:2}
G(\bm y) - H(\bm x) - \inf_{\bm s_H \in \partial H(\bm x)}\langle \bm s_H, \bm y - \bm x \rangle < 0. 
\end{align}
\end{assumption}

\begin{remark}\label{rem:1}
The generalized MFCQ is equivalent to the following condition: For every $\bm x \in \bar{\mathcal{X}}$, there exists $\bd \in \R^n$ such that
\begin{equation}\label{eq:mfcqeq1}
\langle \ba_i, \bd \rangle = 0,\ \text{for all}\ i \in \mE,\quad \langle \nabla \omega_i(\bm x), \bm d \rangle < 0,\ \text{for all}\ i \in \mathcal{A}(\bm x)
\end{equation}
and if $G(\bm x) = H(\bm x)$, it holds that
\begin{equation}\label{eq:mfcqeq2}
G'(\bm x, \bd) -\inf_{\bs_H\in\partial H(\bx)} \langle \bm s_H, \bd \rangle < 0. 
\end{equation}
Please find the detailed proof of this equivalence in \Cref{app:MFCQ}.
\end{remark}


\begin{remark}
Using the equivalence in \Cref{rem:1}, we can derive the condition on $c_i(\cdot,\bm \xi)$ such that the generalized MFCQ holds. Specifically, according to  \eqref{eq:dir der}, \eqref{eq:mfcqeq2} is further equivalent to
\begin{equation}\label{eq:GH}
\sup_{\bs_G\in\partial G(\bx)} \langle \bm s_G, \bd\rangle - \inf_{\bs_H\in\partial H(\bx)} \langle \bm s_H, \bd\rangle <0.
\end{equation}
Using the form of $G$ and $H$ in \eqref{eq:G H} and \Cref{lem:subg H}, we have a complicated representation of \eqref{eq:GH} in forms of gradients of active $c_i(\cdot,\bm\xi)$, which we omitted for simplicity.
In a special case that $ \widehat{C}_{[i]}(\bx)< \widehat{C}_{[i+1]}(\bx)$ for $i=M-1,\ldots,N-1$, we obtain that 
$\partial G(\bx) = \{\sum_{j=M}^{N}\partial\widehat{C}_{i_j}(\bx)\}$
and $\partial H(\bx) =\{ \sum_{j=M+1}^{N}\partial \widehat{C}_{i_j}(\bx)\}$,
where $i_j$ is the index such that $\widehat{C}_{i_j}(\bx)$ is the $j$th smallest element of $\widehat{C}(\bx)$.
We further assume that for every $j=M,\ldots,N$, there is only one active index (say, $l_j$) in $\widehat{C}_{i_j}(\bx)$, i.e., $c_{l_j}(\bx,\hat{\bm \xi}^{i_j})=\widehat{C}_{i_j}(\bx)$ (This holds for single chance constrained programs). 
According to \Cref{lem:subg H},
 $\partial \widehat{C}_{i_j}(\bx)=\{\nabla c_{l_j}(\bx,\hat{\bm\xi}^{i_M})\}$ is a singleton for $j=M,M+1\cdots,N$, and thus \eqref{eq:GH} is equivalent to
\[
 \langle \nabla c_{l_M}(\bx,\hat{\bm\xi}^{i_M}), \bd\rangle < 0.
 \]
\end{remark}

We next introduce the definition of KKT points for Problem \eqref{CCP:sample-DC}. 
\begin{defi}[{\bf KKT Points}]\label{def:KKT}
We say that $\bx \in \bar{\mX}$ is a KKT point of Problem \eqref{CCP:sample-DC} if there exists $\lambda \in \R_+$ such that $(\bx,\lambda)$ satisfies $\lambda\left(G(\bx) - H(\bx)\right) = 0$ and
\begin{align*}
\bm{0} \in \partial g(\bx) - \partial h(\bx) + \lambda\left(\partial G(\bx) - \partial H(\bx)\right) + \mN_{\mX}(\bx).
\end{align*}
\end{defi}
Note that every local minimizer of Problem \eqref{CCP:sample-DC} is a KKT point under the generalized MFCQ. More precisely, suppose that $\bx^* \in \bar\mX$ is a local minimizer of Problem \eqref{CCP:sample-DC}, $\mP=\{\bx:\ba^T\bx+b_i=0,~i\in\mE\}$ is a polyhedron, and there exists $\bm{d} \in \mT_{\mP}(\bx^*)$  for $\bx^*  \in \mX $ such that \eqref{eq:mfcqeq1} and \eqref{eq:mfcqeq2} hold at $\bx^*$. Then, there exists $\lambda^* \in \R_+$ such that $\bx^*$ is a KKT point of Problem \eqref{CCP:sample-DC}. This result is a direct consequence of \cite[Theorem 2.1]{lu2012sequential}. 

\subsection{Subsequential Convergence to a KKT Point}\label{subsec:subseq}

In this subsection, our goal is to show that any accumulation point of the sequence $\{\bx^k\}$ generated by Algorithm \ref{alg-1} is a KKT point of Problem \eqref{CCP:sample-DC}.

\begin{lemma}\label{lem:subseq}
Suppose that \Cref{AS:1} holds, the function $f$ is given in Problem \eqref{CCP:sample-DC}, and the level set $\left\{\bx \in \bar{\mX}:\ f(\bx) \le f(\bx^0)\right\}$ is bounded. Let $\{\bx^k\}$ be the sequence generated by \Cref{alg-1} with  $\rho + 2\beta  > 0$. Then, the following statements hold: \\
(i) It holds for all $k \ge 0$ that $\bx^k \in \bar{\mX}$ and
\begin{align}\label{eq:suff decre}
f(\bx^{k+1}) - f(\bx^k) \le -\frac{\rho+2\beta}{2}\|\bx^{k+1} - \bx^k\|^2.
\end{align}
(ii) The sequence $\{\bx^k\} \subseteq \bar{\mX}$ is bounded.\\
(iii) It holds that
\begin{align}\label{rst:lem subseq}
\lim_{k \rightarrow \infty} \|\bx^{k+1} - \bx^k\| = 0.
\end{align}
\end{lemma}
\begin{proof}
(i) For ease of exposition, let $\mathcal{Y}_k := \left\{ \bm x \in \mathcal{X}:  G(\bx) -  H(\bx^k) - \langle \bs_H^k, \bx-\bx^k \rangle  \le 0 \right\}$ and 
\begin{align*}
    f_k(\bm x) := g(\bm x) - h(\bm x^k) - \langle \bm s_h^k, \bm x - \bm x^k \rangle + \frac{\beta}{2}\|\bm x- \bm x^k\|^2 + \delta_{\mathcal{Y}_k}(\bm x). 
\end{align*}
According to the feasibility of $\bx^{k+1}$ to Problem \eqref{DC:subproblem}, $\bs_H^k \in \partial H(\bx^k)$, and the convexity of $H$, we have $\bx^{k+1} \in \mX$ and
\begin{align}\label{eq1:lem subseq}
G(\bx^{k+1}) \le  H(\bx^k)+ \langle \bs_H^k, \bx^{k+1}-\bx^k \rangle  \le H(\bx^{k+1}).
\end{align}
This implies $\bm x^{k+1} \in \mathcal{Y}_k$ and $\bx^{k+1} \in \bar{\mX}$ for all $k \ge 0$. This further implies $\bm x^k \in \mathcal{Y}_k$. Since $g$ is $\rho$-strongly convex according to \Cref{AS:1}, we have $f_k(\bm x)$ is $(\rho+\beta)$-strongly convex. 
 This, together with $\bm 0 \in \partial f_k(\bx^{k+1})$, $\bm x^k, \bm x^{k+1} \in \mathcal{Y}_k$ and \Cref{lem:sc_lb}, directly yields
 \[
 f_k(\bx^k) \ge f_k(\bx^{k+1}) + \frac{\rho+\beta}{2}\|\bx^{k+1}-\bx^k\|^2,
 \]
 which is equivalent to
\begin{align}
g(\bx^{k+1}) - h(\bx^k) - \langle \bs_h^k, \bx^{k+1} - \bx^k \rangle + \frac{\rho+ 2 \beta}{2}\|\bx^{k+1} - \bx^k\|^2 \le  g(\bx^{k}) - h(\bx^k).
\notag 
\end{align}
This, together with the convexity of $h$ and $\bs_h^k \in \partial h(\bx^k)$, yields that for all $k \ge 0$,
\begin{align*}
g(\bx^{k+1}) - h(\bx^{k+1}) + \frac{\rho+ 2 \beta}{2}\|\bx^{k+1} - \bx^k\|^2 \le  g(\bx^{k}) - h(\bx^k),
\end{align*}
which is equivalent to \eqref{eq:suff decre}.

(ii) According to \eqref{eq:suff decre}, the function value $f(\bx^k)$ is monotonically decreasing and thus we have $f(\bx^{k+1}) \le f(\bx^0)$ for all $ k\ge 1$. This, together with the level-boundedness of the set $\left\{\bx \in \mX^c:\ f(\bx) \le f(\bx^0)\right\}$, implies that $\{\bx^k\}$ is bounded.

(iii) 
The boundedness of the sequence $\{\bx_k\}$, together with continuity of $f$,
implies that $\{f(\bx^k)\}$ is bounded from below. Using this and the fact that $\{f(\bx^k)\}$ is monotonically decreasing, we obtain that there exists some $f^*$ such that $f(\bx^k) \rightarrow f^*$. It follows from \eqref{eq:suff decre} that
\begin{align*}
\frac{\rho+2\beta}{2}\sum_{k=0}^\infty \|\bx^{k+1} - \bx^k\|^2 \le f(\bx^0) - \lim_{k \rightarrow \infty} f(\bx^{k+1}) = f(\bx^0) - f^* < \infty.
\end{align*}
This implies \eqref{rst:lem subseq}. 
\end{proof}

Armed with the above lemma, we are ready to show the subsequential convergence of the sequence $\{\bx^k\}$ generated by Algorithm \ref{alg-1} to a KKT point of Problem \eqref{CCP:sample-DC}.
\begin{thm}\label{thm:subseq}
Suppose that Assumptions \ref{AS:1} and \ref{AS:MFCQ} hold and the level set $\left\{\bx \in \bar{\mX}:\ f(\bx) \le f(\bx^0)\right\}$ is bounded. Let $\{\bx^k\}$ be the sequence generated by Algorithm \ref{alg-1}  with $\rho+ 2 \beta > 0$. Then, any accumulation point of $\{\bx^k\}$ is a KKT point of Problem \eqref{CCP:sample-DC}.
\end{thm}
\begin{proof}
According to (i) in Lemma \ref{lem:subseq}, it holds that $\bx^k \in \bar{\mX}$ for all $k \ge 0$. Using the generalized MFCQ in Assumption \ref{AS:MFCQ}, there exists $\bm x \in \mathcal{X}$ such that 
\begin{align}
    & \langle \nabla \omega_i(\bm x^k), \bm x - \bm x^k \rangle < 0,\ \forall i \in \mathcal{A}(\bm x^k), \quad\text{ and }\label{eq1:thm subseq}\\
    & G(\bx) - H(\bx^k) - \langle \bs_H^k,\bx-\bx^k \rangle < 0,\ \text{if}\ G(\bm x^k) = H(\bm x^k). \label{eq2:thm subseq}
\end{align}
According to \eqref{set:active}, we have $\omega_i(\bm x^k) = 0$ for all $i \in \mathcal{A}(\bm x^k)$. Let $\bm x_{\alpha} := \alpha \bm x + (1-\alpha)\bm x^k$, where $\alpha \in (0,1]$. Obviously, we have $ \langle \nabla \omega_i(\bm x^k), \bm x_\alpha - \bm x^k \rangle = \alpha  \langle \nabla \omega_i(\bm x^k), \bm x - \bm x^k \rangle < 0$ for all $i \in \mathcal{A}(\bm x^k)$. Since $\omega_i(\bm x)$ is continuously differentiable, its Taylor expansion at $\bm x^k$ is as follows:
\begin{align*}
    \omega_i\left( \bm x_\alpha \right) =   \omega_i(\bm x^k) + \langle \nabla  \omega_i(\bm x^k), \bm x_{\alpha} - \bm x^k \rangle + o\left( \|\bm x_{\alpha} - \bm x^k\| \right).
\end{align*}
This, together with \eqref{eq1:thm subseq}, $\omega_i(\bm x^k) = 0$ for all $i \in \mathcal{A}(\bm x^k)$, and $\|\bm x_{\alpha} - \bm x^k \|= \alpha\| \bm x-\bm x^k\|$, yields $\omega_i\left( \bm x_\alpha \right) < 0$ for all $i \in \mathcal{A}(\bm x^k)$ when $\alpha\to 0$. Moreover, according to \eqref{eq2:thm subseq}, if $G(\bm x^k) = H(\bm x^k)$, we have 
\begin{align*} 
    G(\bx_\alpha) - H(\bx^k) - \langle \bs_H^k,\bx_\alpha - \bx^k \rangle & =  G(\alpha \bm x + (1-\alpha)\bm x^k)  - H(\bx^k) - \alpha \langle \bs_H^k, \bm x - \bx^k \rangle \\
    & \le \alpha \left( G(\bm x) - H(\bm x^k) - \langle \bs_H^k, \bm x - \bx^k \rangle \right)  + (1-\alpha) \left( G(\bm x^k) - H(\bm x^k)\right) \\
    & = \alpha \left( G(\bm x) - H(\bm x^k) - \langle \bs_H^k, \bm x - \bx^k \rangle \right)  < 0,
\end{align*}
where the first inequality uses the convexity of $G$, the second equality follows from $G(\bm x^k) = H(\bm x^k)$, and the last inequality is due to \eqref{eq2:thm subseq}. Using the similar argument, when $\alpha > 0$ is sufficiently small, it follows from $\omega_i(\bm x^k) < 0$ for all $i \in \mathcal{I} \setminus \mathcal{A}(\bm x^k)$ due to \eqref{set:active} and  $G(\bm x^k)-H(\bm x^k) < 0$ that  
\begin{align*} 
\omega_i(\bx_\alpha) < 0,\ \forall i \in \mathcal{I} \setminus \mA(\bx^k),\ G(\bx_\alpha) - H(\bx^k) - \langle \bs_H^k,\bx_\alpha-\bx^k \rangle <0,\ \text{if}\ G(\bm x^k) < H(\bm x^k).
\end{align*}
Therefore, we obtain that there exists $\by \in {\mX}$ such that for any $\bs_H^k \in \partial H(\bx^k)$, 
\begin{align}\label{CQ:Slater}
\omega_i(\by) < 0,\ \forall i \in \mathcal{I},\ G(\by) - H(\bx^k) - \langle \bs_H^k,\by-\bx^k \rangle <0.
\end{align}
This is exactly the Slater condition for Problem \eqref{DC:subproblem}. Consequently, according to \cite[Theorem 28.2]{rockafellar1970convex}, there exists a Lagrange multiplier $\lambda^{k} \in \R$ associated with the constraint $G(\bx^{k+1}) - H(\bx^k) - \langle \bs_H^k,\bx^{k+1}-\bx^k\rangle \le 0$ such that the following KKT system holds:
\begin{align}\label{eq2:lem-decrease}
\begin{cases}&G(\bx^{k+1}) - H(\bx^k) - \langle \bs_H^k,\bx^{k+1}-\bx^k\rangle \le 0,\ \\
&\lambda^k \left(G(\bx^{k+1}) - H(\bx^k) - \langle \bs_H^k,\bx^{k+1}-\bx^k\rangle\right) = 0,\\
& \bm{0} \in \partial g(\bx^{k+1}) - \bs_h^k + \beta(\bx^{k+1}-\bx^k) + \lambda^k \left(\partial G(\bx^{k+1}) - \bs_H^k\right) + \mN_{\mX}(\bx^{k+1}),\\
&\bx^{k+1} \in \mX, \lambda^k \ge 0.\end{cases}
\end{align}
It follows from (ii) of Lemma \ref{lem:subseq} that $\{\bx^k\}$ is bounded. Let $\bx^*$ be an accumulation point of $\{\bx^k\}$ such that there exists a subsequence $\{\bx^{k_i}\}$ with $\lim_{i \rightarrow \infty} \bx^{k_i} = \bx^*$. We claim that the sequence $\{\lambda^k\}$ is bounded. Passing to a further subsequence if necessary, we assume without loss of generality that $\lim_{i\rightarrow \infty} \lambda^{k_i} = \lambda^*$. According to \eqref{rst:lem subseq} in Lemma \ref{lem:subseq}, we have $\lim_{i\rightarrow \infty} (\bx^{k_i+1}-\bx^{k_i}) = \bm{0}$. Using this fact, the outer semi-continuity of $\partial g$, $\partial h$, $\partial G$, $\partial H$, and the normal cones of convex closed sets (see \Cref{def:osc} and \Cref{lem:osc}), and $\bs_h^k \in \partial h(\bx^k),\ \bs_H^k \in \partial H(\bx^k)$, we obtain upon passing to the limit as $i$ goes to infinity in \eqref{eq2:lem-decrease} with  $k=k_i$ that
$\bs_h^k\to\bs_h^*\in\partial h(\bx^*)$ and $\bs_H^k\to\bs_H^*\in\partial H(\bx^*)$, and thus
\begin{align}\label{eq4:lem-decrease}
\bm{0} \in \partial g(\bx^*) - \partial h(\bx^*) + \lambda^* \left(\partial G(\bx^*) - \partial H(\bx^*)\right) + \mN_{\mX}(\bx^{*}).
\end{align}
On the other hand, using \eqref{eq2:lem-decrease} and \eqref{rst:lem subseq} with $k=k_i$ and the boundedness of $\partial H(\bx^*)$, letting $i\rightarrow \infty$, we have
\begin{align}
& G(\bx^*) \le H(\bx^*),\ \lambda^*\left( G(\bx^{*}) - H(\bx^*) \right) = 0.\label{eq5:lem-decrease}
\end{align}
Moreover, since $\lambda^k \ge 0$ and $\bx^k \in \bar{\mX}$ for all $k \ge 0$, we have $\lambda^* \ge 0$ and $\bx^* \in \bar{\mX}$. This, together with \eqref{eq4:lem-decrease}, \eqref{eq5:lem-decrease}, and \Cref{def:KKT}, implies that $\bx^*$ is a KKT point of Problem \eqref{CCP:sample-DC}.

The rest of the proof is devoted to proving that $\{\lambda^k\}$ is bounded. Without loss of generality, we assume that $\{\ba_i: i\in \mE\}$ is linearly independent, since otherwise, we can obtain the same results by eliminating the redundant linear equalities. It follows from Lemma \ref{lem:normal cone} that  for any $\bx \in \mX$, 
\begin{align*}
\mN_{\mX}(\bx) = \left\{ \sum_{i\in \mE}u_i\ba_i + \sum_{i\in \mI}v_i\nabla \omega_i(\bx): \ v_i \ge 0,\ \text{for}\ i \in \mA(\bx),\ v_i = 0,\ \text{for}\ i \in \mathcal{I} \setminus \mathcal{A}(\bm x)  \right\}.
\end{align*}
This, together with \eqref{eq2:lem-decrease}, yields that there exist $u_i^{k}$ for $i \in \mE$, $v_i^{k} \ge 0$ for $i \in \mA(\bx^{k+1})$, and $v_i^{k} = 0$ for $i \in \mathcal{I} \setminus \mA(\bx^{k+1})$ such that
\begin{align}\label{eq7:lem-decrease}
\bm{0} \in \ \partial g(\bx^{k+1}) - \bs_h^{k} + \beta(\bx^{k+1}-\bx^{k}) + \lambda^{k} \left(\partial G(\bx^{k+1}) - \bs_H^{k} \right) + \sum_{i\in \mE}u_i^{k}\ba_i + \sum_{i\in \mI}v_i^{k} \nabla \omega_i(\bx^{k+1}).
\end{align}
Then, let
\begin{align*}
\rho^k := \sqrt{(\lambda^k)^2 + \sum_{i\in \mE} (u_i^k)^2 + \sum_{i \in \mI} (v_i^k)^2}, \tau^k := \frac{\lambda^{k}}{\rho^k},\ \mu_i^k := \frac{u_i^{k}}{\rho^k},\ \nu_i^k := \frac{v_i^{k}}{\rho^k}.
\end{align*}
Suppose to the contrary that $\{\lambda^k\}$ is unbounded. This implies that ${\rho^k}$ is also unbounded.
Then, there exists a subsequence $\{\lambda^{k_j}\}$ such that $|\lambda^{k_j}| \rightarrow \infty$ as $j$ goes to infinity. Passing to a further subsequence if necessary, suppose that there exist $\tau^* \in \R_+$, $\mu_i^* \in \R$, $\nu_i^* \in \R_+$, $\bx^*$, and $\bs_H^* \in \partial H(\bx^*)$ such that $\lim_{j\rightarrow \infty} \tau^{k_j}  = \tau^*$, $\lim_{j\rightarrow \infty} \mu_i^{k_j}  = \mu_i^*$, $\lim_{j\rightarrow \infty} \nu_i^{k_j}  = \nu_i^*$, $\lim_{j\rightarrow \infty} \bx^{k_j} = \bx^*$, and $\lim_{j\rightarrow \infty} \bs_H^{k_j} = \bs_H^*$, where $\bs_H^{k_j} \in \partial H(\bx^{k_j})$, due to $\lambda^{k} \ge 0$, $v_i^* \ge 0$ for $i \in \mI$, the boundedness of $\{\tau^k\}$, $\{\bmu^k\}$, $\{\bnu^k\}$, $\{\bx^k\}$, and $\partial H(\bx^k)$, and the outer semi-continuity of $\partial H$.
Then, dividing both sides of \eqref{eq7:lem-decrease} by $|\rho^{k_j}|$, letting $j \rightarrow \infty$, and using \eqref{rst:lem subseq}, the outer semi-continuity of $\partial g$ and $\partial h$, and the boundedness of $\partial g(\bx^{*})$, $\partial h(\bx^{*})$, and $\{\bx^k\}$, we have
\begin{align}\label{eq6:lem-decrease}
\bm{0} \in \tau^{*} \left(\partial G(\bx^{*}) - \bs_H^{*} \right) + \sum_{i\in \mE} \mu_i^*\ba_i + \sum_{i\in \mI} \nu_i^*\nabla \omega_i(\bx^*).
\end{align}
Using the definitions of $\tau^*,\bmu^*$, and $\bnu^*$, we further have
\begin{align}\label{eq:sumone}
(\tau^*)^2+\|\bmu^*\|^2+\|\bnu^*\|^2=1,
\end{align}
({\bf Case 1}) Suppose that $\tau^* = 0$. Due to \eqref{eq6:lem-decrease}, we have
\begin{align}\label{eq8:lem-decrease}
\bm{0} = \sum_{i\in \mE} \mu_i^*\ba_i + \sum_{i\in \mI} \nu_i^*\nabla \omega_i(\bx^*).
\end{align}
According to \eqref{slater:1} in Assumption \ref{AS:MFCQ}, there exists $\by \in \mX$ such that $\langle \nabla \omega_i(\bx^*),\by-\bx^*\rangle < 0$ for all $i \in \mA(\bx^*)$. Moreover, since $\mA(\bx^k) \subseteq \mA(\bx^*)$ when $k$ is sufficiently large, we have $i \notin \mA(\bx^k)$ if $i \notin \mA(\bx^*)$. Therefore, we have $\nu_i^k = 0$ for all $i \notin \mA(\bx^{k+1})$ as $k \rightarrow \infty$, which implies $\nu_i^* = 0$ for $i \notin \mA(\bx^*)$.
Then, taking inner products with $\by-\bx^*$ on both sides of \eqref{eq8:lem-decrease} yields
\begin{align*}
0 = \sum_{ i \in \mA(\bx^*)} \nu_i^* \langle \nabla \omega_i(\bx^*), \by-\bx^* \rangle,
\end{align*}
where the equality follows from $\langle \ba_i,\by-\bx^* \rangle = 0$ for $i \in \mE$ and $\nu_i^* = 0$ for $i \notin \mA(\bx^*)$. This, together with $\langle \nabla \omega_i(\bx^*),\by-\bx^*\rangle < 0$ for all $i \in \mA(\bx^*)$, gives $\nu_i^* = 0$ for all $i \in \mA(\bx^*)$. Substituting this and $\nu_i^* = 0$ for $i \notin \mA(\bx^*)$ into \eqref{eq8:lem-decrease}, we have $\bm{0} = \sum_{i\in \mE} \mu_i^*\ba_i$. Noting that we assume that $\{\ba_i: i\in \mE\}$ is linearly independent, we have $\mu_i^*=0$ for all $i \in \mE$. Therefore, $\nu_i^*=0$ for all $i \in \mI$ and $\mu_i^*=0$ for all $i \in \mE$. This contradicts \eqref{eq:sumone}. \\
({\bf Case 2}) Suppose that $\tau^* > 0$. We first consider the case of $G(\bx^*) < H(\bx^*)$. It follows from the second line of \eqref{eq2:lem-decrease} with $k=k_j$, $j\rightarrow\infty$, and \eqref{rst:lem subseq} that $\lim_{j\rightarrow \infty} \lambda^{k_j} = 0$. This implies $\tau^*=0$, which contradicts $\tau^* > 0$. We then must have $G(\bx^*)=H(\bx^*)$. This, together with the convexity of $G$ and \eqref{slater:2} in Assumption \ref{AS:MFCQ},  yields that there exists $\by \in \mX$ such that
\begin{align} \label{eq9:lem-decrease}
\langle \bar\bs_G - \bs_H^{*}, \by-\bx^* \rangle & \le G(\by) - G(\bx^*) - \langle \bs_H^{*}, \by-\bx^* \rangle \notag\\
& = G(\by) - H(\bx^*) - \langle \bs_H^*,\by-\bx^* \rangle < 0,
\end{align}
where $\bar\bs_G$ is an arbitrary subgradient of $G$ at $\bx^*$.
According to \eqref{eq6:lem-decrease}, there exists $\bs_G^* \in \partial G(\bx^*)$ such that
\begin{align}
\bm{0} = \tau^{*} \left(\bs_G^* - \bs_H^{*} \right) + \sum_{i\in \mE} \mu_i^*\ba_i + \sum_{i\in \mI} \nu_i^*\nabla \omega_i(\bx^*).
\end{align}
Taking inner products with $\by-\bx^*$ on both sides yields
\begin{align*}
0 = \tau^* \langle \bs_G^* - \bs_H^{*}, \by-\bx^* \rangle + \sum_{i\in \mA(\bx^*)} \nu_i^*\langle \nabla \omega_i(\bx^*), \by-\bx^* \rangle.
\end{align*}
Note that $\nu_i^* \ge 0$ due to $v_i^k \ge 0$ for all $i \in \mI$. This, together with \eqref{slater:2} at $\bx^*$ and \eqref{eq9:lem-decrease}, implies $\tau^*=0$, which is a contradiction. We prove the claim.
\end{proof}
 
\subsection{Convergence of the Entire Sequence to a KKT Point}

In this subsection, we employ the analytical framework proposed in \cite{attouch2009convergence,attouch2013convergence} based on the K\L\ property to study the sequential convergence of Algorithm \ref{alg-1} for $\beta + 2\rho > 0$. Our first step is to show that the sequence generated by Algorithm \ref{alg-1} satisfies \emph{sufficient decrease} and \emph{relative error} conditions with respect to a potential function. 
Motivated by the potential functions constructed in \cite{liu2019refined,yu2021convergence}, we construct the following potential function
\begin{align}\label{poten:phi}
\varphi(\bx,\by,\bz) := g(\bx) - \langle \bx, \by \rangle + h^*(\by) + \delta_{\bar{F}(\cdot) \le 0}(\bx,\bz) + \delta_{\mX}(\bx),
\end{align}
where
\begin{align}\label{poten:F}
\bar{F}(\bx,\bz) := G(\bx) - \langle \bx,\bz \rangle + H^*(\bz).
\end{align}
Then, we characterize the subdifferential of $\delta_{\bar{F}(\cdot) \le 0}(\bx,\bz)$ using its structure and the convexity of $G$ and $H$. Notably, this characterization holds for $G$ and $H$ being arbitrary proper closed convex functions.
\begin{lemma}\label{lem:subg F}
Suppose that \Cref{AS:MFCQ} holds and $(\bx,\bz)$ satisfies $\bar{F}(\bx,\bz) \le 0$ and $\bm x \in \mathcal{X}$. It holds that 
\begin{align}\label{subg:F}
    \widehat{\partial}  \delta_{\bar{F}(\cdot) \le 0}(\bx,\bz) \supseteq  \left\{\begin{bmatrix}
\lambda(\partial G(\bx) - \bz)  \\
 \lambda(-\bx + \partial H^*(\bz))
\end{bmatrix}: \lambda \ge 0,\ \lambda\bar{F}(\bx,\bz)=0 \right\}.
\end{align}
\end{lemma}
\begin{proof}
To begin, let
\begin{align*}
\mS:=\left\{(\bx,\bz): \bar{F}(\bx,\bz) \le 0\right\}.
\end{align*}
In addition, we write $\mS=\bar{F}^{-1}(\R_-)$.
Because $G$ and $H^*$ are both convex functions, then $\bar F$ is locally Lipschitz continuous.
This, together with \Cref{def:cts} and \Cref{lem:cts}, implies that $\bar{F}: \R^n\times \R^n \rightarrow \R$ is a strictly continuous function. Using this and \Cref{lem:chain}, we obtain 
\begin{align*}
    \widehat{\mathcal{N}}_{\cal S}(\bm x, \bm z) \supseteq \left\{\widehat{\partial} (\lambda\bar{F})(\bm x,\bm z): \lambda \in \widehat{\mathcal{N}}_{\R_{-}}(\bar{F}(\bm x, \bm z)) \right\}.  
\end{align*}
Since $\mN_{\R_-}(\bar{F}(\bx,\bz)) = \widehat{\mathcal{N}}_{\R_{-}}(\bar{F}(\bm x, \bm z))$ due to the convexity of $\R_{-}$, then $\lambda \in  \widehat{\mN}_{\R_-}(\bar{F}(\bx,\bz))$ is equivalent to $\lambda \ge 0,\ \lambda\bar{F}(\bx,\bz)=0$. According to (ii) and (iii) of \Cref{lem:rule sub}, we obtain 
\begin{align*}
    \widehat{\partial} (\lambda\bar{F})(\bm x,\bm z)  \supseteq \lambda\begin{bmatrix}
 \widehat{\partial} G(\bx)  \\
 \widehat{\partial} H^*(\bz)
\end{bmatrix} + \lambda \begin{bmatrix}
    -\bm z\\ -\bm x
\end{bmatrix} = \begin{bmatrix}
\lambda(\partial G(\bx) - \bz)  \\
 \lambda(-\bx + \partial H^*(\bz))
\end{bmatrix},
\end{align*}
where the equality follows from the convexity of $G(\cdot)$ and $H(\cdot)$. These, together with $\widehat{\cal N}_{\cal S}(\bm x, \bm z) = \widehat{\partial} \delta_{\bar F(\cdot) \le 0}(\bm x, \bm z)$, yield \eqref{subg:F}.  
\end{proof}
 
Now, we are ready to show that the sequence $\{(\bx^k,\bs_h^k,\bs_H^k)\}$ generated by Algorithm \ref{alg-1} satisfies the sufficient decrease and relative error conditions mentioned earlier. 

\begin{lemma}\label{lem:suff-safe}
Suppose that Assumptions \ref{AS:1} and \ref{AS:MFCQ} hold. Let $\{(\bx^{k+1},\bs^k_h,\bs_H^{k})\}$ be the sequence generated by Algorithm \ref{alg-1}  with $\rho+2\beta > 0$. Then, the following statements hold:\\
(i) \emph{[Sufficient Decrease]} The sequence $\{(\bx^{k+1},\bs^k_h,\bs_H^{k})\}$ is bounded. It holds for all $k \ge 1$ that
\begin{align*}
\varphi(\bx^{k+1},\bs^k_h,\bs_H^{k}) - \varphi(\bx^{k},\bs^{k-1}_h,\bs_H^{k-1}) \le -\frac{\rho+2\beta}{2}\|\bx^{k+1} - \bx^k\|^2.
\end{align*}
(ii) \emph{[Relative Error]} There exists a constant $\kappa > 0$ such that for all $k \ge 0$,
\begin{align*}
\dist\left(\bm{0},\partial \varphi(\bx^{k+1},\bs_h^k,\bs_H^k)\right) \le \kappa \|\bx^{k+1} - \bx^k\|.
\end{align*}
\end{lemma}
\begin{proof}
(i) It follows from (i) in Lemma \ref{lem:subseq} that $\{\bx^k\} \subseteq \bar{\mX}$ is bounded. This, together with the fact that $h$ and $H$ are convex, implies that $\{(\bs^k_h,\bs_H^{k})\}$ is bounded. Therefore, the sequence $\{(\bx^{k+1},\bs^k_h,\bs_H^{k})\}$ is bounded. According to \eqref{poten:F}, we have for all $k \ge 0$,
\begin{align}\label{eq1:lem suff safe}
\begin{aligned}
\bar{F}(\bx^{k+1},\bs_H^k) & = G(\bx^{k+1}) - \langle \bx^{k+1},\bs_H^k \rangle + H^*(\bs_H^k) \\
& = G(\bx^{k+1}) + H^*(\bs_H^k) - \langle \bx^k,\bs_H^k \rangle - \langle \bx^{k+1} - \bx^k, \bs_H^k \rangle \\
& = G(\bx^{k+1}) - H(\bx^k) - \langle \bx^{k+1} - \bx^k, \bs_H^k \rangle \le 0,
\end{aligned}
\end{align}
where the last equality follows from $H(\bx^k) + H^*(\bs_H^k) = \langle\bx^k, \bs_H^k \rangle$ due to Young's inequality and $\bs_H^k \in \partial H(\bx^k)$, and the inequality is due to the constraint in \eqref{DC:subproblem}.
Moreover, it follows from \eqref{DC:subproblem}, the $(\rho+\beta)$-strongly convexity of $g(x)- \langle \bs_h^k, \bx - \bx^k \rangle + {\beta}\|\bx - \bx^k\|^2/2$, and Lemma \ref{lem:sc_lb} that for all $ k \ge 0$,
\begin{align}\label{eq2:lem suff safe}
g(\bx^{k+1}) - \langle \bs_h^k, \bx^{k+1} - \bx^k \rangle + \frac{\rho+2\beta}{2}\|\bx^{k+1} - \bx^k\|^2 \le g(\bx^k).
\end{align}
This, together with \eqref{eq1:lem suff safe} and $\bx^k \in \mX$, implies for all $ k \ge 1$,
\begin{align*}
\varphi(\bx^{k+1},\bs_h^k,\bs_H^k) & = g(\bx^{k+1}) - \langle \bx^{k+1}, \bs_h^k \rangle + h^*(\bs_h^k)  \\
& \le g(\bx^k) - \langle \bs_h^k, \bx^k \rangle - \frac{\rho+2\beta}{2}\|\bx^{k+1} - \bx^k\|^2  + h^*(\bs_h^k)  \\
& = g(\bx^k) - h(\bx^k)  -  \frac{\rho+2\beta}{2}\|\bx^{k+1} - \bx^k\|^2 \\
& \le g(\bx^k) - \langle \bx^k, \bs_h^{k-1} \rangle + h^*(\bs_h^{k-1}) - \frac{\rho+2\beta}{2}\|\bx^{k+1} - \bx^k\|^2 \\
& = \varphi(\bx^{k},\bs_h^{k-1},\bs_H^{k-1}) - \frac{\rho+2\beta}{2}\|\bx^{k+1} - \bx^k\|^2,
\end{align*}
where the first inequality uses \eqref{eq2:lem suff safe}, the second equality follows from $h(\bx^k)+h^*(\bs_h^k)=\langle \bx^k,\bs_h^k \rangle$ due to $\bs_h^k \in \partial h(\bx^k)$ and Young's inequality, the second inequality follows from $h(\bx^k)+h^*(\bs_h^{k-1}) \ge \langle \bx^k,\bs_h^{k-1} \rangle$ due to Young's inequality, and the last equality is due to $\bx^k \in \mX$, \eqref{poten:phi}, and \eqref{eq1:lem suff safe}. \\
(ii) To begin, we compute 
\begin{align}\label{eq5:lem suff safe}
\partial \varphi(\bx,\by,\bz) & \supseteq \widehat{\partial} \varphi(\bx,\by,\bz) \supseteq \begin{bmatrix}
\widehat{\partial}g(\bx) - \by + \widehat{\partial} \delta_{\mX}(\bx) \\
-\bx + \widehat{\partial} h^*(\by)  \\
\bm{0} 
\end{bmatrix} + \mathcal{B} = \begin{bmatrix}
\partial g(\bx) - \by + \mN_{\mX}(\bx) \\
-\bx + \partial h^*(\by) \\
\bm{0} 
\end{bmatrix} + \mathcal{B},
\end{align}
where the first inclusion follows from (i) of \Cref{lem:rule sub}, the second inclusion uses (ii), (iii), and (iv) of \Cref{lem:rule sub} and $\mathcal{B}: = \{(\bm x, \bm y, \bm z)\in \R^n\times \R^n\times \R^n: (\bm x, \bm z) \in \widehat{\partial}\delta_{\bar{F}(\cdot) \le 0}(\bx,\bz), \bm y = \bm 0\}$, and the equality is due to the convexity of $g$, $h^*$, and $\mX$ and the fact that $\widehat{\partial} f(\bm x) = \partial f(\bm x)$ for any proper and convex function $f$ and $\bm x \in \mathrm{dom}(f)$. According to \Cref{lem:subg F}, we obtain 
\begin{align*}
\widehat{\partial}\delta_{\bar{F}(\cdot) \le 0}(\bx,\bz) \supseteq \left\{  
\begin{bmatrix}
\lambda(\partial G(\bx) - \bz)  \\
 \lambda(\partial H^*(\bz)-\bx)
\end{bmatrix}: \lambda \ge 0,\ \lambda\bar{F}(\bm x, \bm z) = 0 \right\}
\end{align*}
This, together with \eqref{eq5:lem suff safe}, implies 
\begin{equation}\label{eq3:lem suff safe}
\partial \varphi(\bx^{k+1},\bs_h^k,\bs_H^k) \supseteq \left\{\begin{bmatrix}
\partial g(\bx^{k+1}) - \bs_h^k + \mN_{\mX}(\bx^{k+1}) + \lambda(\partial G(\bx^{k+1}) - \bs_H^k )\\
-\bx^{k+1} + \partial h^*(\bs_h^k)\\
\lambda(\partial H^*(\bs_H^k)-\bx^{k+1})
\end{bmatrix}:  \lambda \ge 0,\ \lambda\bar{F}(\bm x^{k+1}, \bm s_{H}^k) = 0\right\}. 
\end{equation}
It follows from \Cref{AS:MFCQ} that the KKT system \eqref{eq2:lem-decrease} holds for Problem \eqref{DC:subproblem}. Then we have $\lambda^k\ge 0$ and
\begin{equation}\label{eq4:lem suff safe}
\begin{aligned}
\lambda^k\bar{F}(\bx^{k+1},\bs_H^k) & = \lambda^k\left(G(\bx^{k+1}) - \langle \bx^{k+1}, \bs_H^k\rangle + H^*(\bs_H^k)\right)  \\
& = \lambda^k\left(G(\bx^{k+1}) - H(\bx^k) - \langle \bx^{k+1} - \bx^{k}, \bs_H^k\rangle\right) = 0,
\end{aligned}
\end{equation}
where the first equality uses \eqref{poten:F}, the second equality follows from $H(\bx^k)+H^*(\bs_H^k)=\langle \bx^k,\bs_H^k \rangle$ due to $\bs_H^k \in \partial H(\bx^k)$ and Young's inequality, and the last equality is due to the second line in \eqref{eq2:lem-decrease}. It follows from the last line in \eqref{eq2:lem-decrease} that
\begin{align*}
\beta(\bx^k - \bx^{k+1}) \in \partial g(\bx^{k+1}) - \bs_h^k + \lambda^k \left(\partial G(\bx^{k+1}) - \bs_H^k\right) + \mN_{\mX}(\bx^{k+1}).
\end{align*}
This, together with \eqref{eq1:lem suff safe}, \eqref{eq3:lem suff safe}, \eqref{eq4:lem suff safe} with $\lambda^k \ge 0$, $\bs_h^k \in \partial h(\bx^k)$, $\bs_H^k \in \partial H(\bx^k)$, and the fact that $\by \in \partial \psi(\bx)$ if and only if $\bx \in \partial \psi^*(\by)$ provided that $\psi$ is a proper closed convex function, yields that
\begin{align*}
\left(\beta(\bx^k - \bx^{k+1}),\ \bx^k - \bx^{k+1},\ \lambda^k(\bx^k - \bx^{k+1})\right) \in \partial \varphi(\bx^{k+1},\bs_h^k,\bs_H^k).
\end{align*}
This implies
\begin{align*}
\dist\left(\bm{0},\partial \varphi(\bx^{k+1},\bs_h^k,\bs_H^k)\right) \le (\beta + 1 +\lambda^k) \|\bx^{k+1} - \bx^k\|,
\end{align*}
where $\lambda^k \ge 0$ is bounded in \eqref{eq2:lem-decrease} according to the proof of \Cref{thm:subseq}.
 \end{proof}

To apply the K\L\ property to conduct convergence analysis, we require that the function $\varphi$ is a K\L\ function. According to \cite[Theorem 3 \& Example 2]{bolte2014proximal} and \cite[Section 4.3]{attouch2010proximal}, if $\varphi$ is proper, lower semicontinuous, and semialgebraic (see \Cref{def:semi}), then $\varphi$ satisfies the K\L\ property on $\mathrm{dom}(\varphi)$. According to \Cref{AS:1} and \Cref{exam:semi} in \Cref{app:D}, the following conditions suffice to guarantee $\varphi$ to be a K\L\ function: The functions $g,h$ are semialgebraic, $c_i(\bm x, \bm \xi)$ for all $i \in \{1,\dots,m\}$  semialgebraic in $\bm x$ for every $\bm \xi \in \Xi$, and $\omega_i$ for all $i \in \mathcal{I}$ are semialgebraic. Using Lemma \ref{lem:suff-safe} and the analysis in \cite{attouch2009convergence,attouch2010proximal,attouch2013convergence,bolte2014proximal,liu2019refined,yu2021convergence}, one can prove the following result on the sequential convergence and the convergence rate of the sequence $\{\bx^k\}$ generated by Algorithm \ref{alg-1}. The proof is rather standard and thus we omit it. We refer the reader to \cite{attouch2009convergence,liu2019refined} for the detailed arguments.

\begin{thm}\label{thm:entire conv}
Let the function $f$ be defined in \Cref{AS:1}. Suppose that the level set $\left\{\bx \in \mX^c:\ f(\bx) \le f(\bx^0)\right\}$ is bounded, $\varphi$ in \eqref{poten:phi} is a K\L\ function with exponent $\theta \in [0,1)$, and Assumptions \ref{AS:1} and \ref{AS:MFCQ} hold. Then, the sequence $\{\bx^k\}$ generated by Algorithm \ref{alg-1} with $\rho + 2\beta > 0$ converges to a KKT point $\bx^*$ of Problem \eqref{CCP:sample-DC}. There exists an integer $k^* \ge 1$ such that the following statements hold: \\
(i) If $\theta=0$, then $\{\bx^k\}$ converges finitely, i.e., $\bx^k=\bx^*$ for all $k \ge k^*$. \\
(ii) If $\theta \in (0,1/2]$, then $\{\bx^k\}$ converges linearly, i.e., there exist $c>0$ and $q \in (0,1)$ such that for all $k \ge k^*$,
\begin{align*}
\|\bx^k - \bx^*\| \le cq^k.
\end{align*}
(iii) If $\theta \in (1/2,1)$, then $\{\bx^k\}$ converges sublinearly, i.e.,  there exist $c>0$ such that for all $k \ge k^*$,
\begin{align*}
\|\bx^k - \bx^*\| \le ck^{-\frac{1-\theta}{2\theta-1}}.
\end{align*}
\end{thm}
It follows from Theorem \ref{thm:entire conv} that the proximal DC algorithm achieves linear convergence when the K\L\ exponent $\theta=1/2$. Therefore, an interesting future direction is to investigate under what conditions the K\L\ exponent of Problem \eqref{CCP:sample-DC} is $1/2$; see, e.g., \cite{li2018calculus,jiang2019novel,jiang2022holderian,liu2019quadratic,wang2021linear,zheng2022linearly}.

\subsection{Iteration Complexity for Computing an Approximate KKT Point}

In this subsection, we analyze the iteration complexity of Algorithm \ref{alg-1} for computing an approximate KKT point of Problem \eqref{CCP:sample-DC}. Motivated by the analysis framework in \cite{yurtsever2022cccp} for DC constrained DC programs with all functions being differentiable, we connect Algorithm \ref{alg-1} to a variant of the Frank-Wolfe (FW) method. To simplify notation, let
\[
\bw:=(\bx,s,t),\quad q(\bw):=s-h(\bx),\quad Q(\bw):=t-H(\bx),\] and
\[\mW:=\left\{\bw: \bx \in \mX,\ g(\bx)\le s,\ G(\bx) \le t \right\}.\]
In particular, we should mention that $q$ and $Q$ are both concave functions and $\mW$ is a convex set.
We rewrite Problem \eqref{CCP:sample-DC} as follows by introducing auxiliary variables $s,t \in \R$:
\begin{align}
\begin{aligned}\label{P:FW DC}
\min_{\bx \in \mX, s \in \R, t \in \R} &  s - h(\bx) \\
\st\quad ~~&\ g(\bx) \le s,\ G(\bx) \le t,\ t - H(\bx) \le 0.
\end{aligned}
\end{align}
We further express Problem \eqref{P:FW DC} as
\begin{align}\label{P:FW}
\min_{\bm{w} \in \mW}\ q(\bw)\quad \st\ Q(\bw) \le 0,
\end{align}
Based on the above setup, together with defining $\|\bz\|_T=\sqrt{\sum_{i=1}^{n}z_i^2}$ for any $\bz\in\R^{n+2}$, we directly show the equivalence between the proximal DC iterations in \eqref{DC:subproblem} and a variant of FW iterations applied to Problem \eqref{P:FW}.
\begin{lemma}\label{lem:equi DC FW}
The proximal DC iterations in \eqref{DC:subproblem} with $\beta  \ge 0$ is equivalent to the following variant of FW iterations:
\begin{align}
\begin{aligned}\label{subP:FW}
\bw^{k+1} \in  \argmin_{\bw \in \mW}~& q(\bw^k) + \langle \bs_q^k,\bw-\bw^k \rangle +  \frac{\beta}{2}\|\bw-\bw^k \|_T^2 \\
\rm s.t.\quad& Q(\bw^k) + \langle \bs_Q^k,\bw-\bw^k \rangle \le 0,
\end{aligned}
\end{align}
where $\bs_q^k=(-\bs_h^k, 1,0)$, $\bs_h^k\in \partial h(\bx^k)$, $\bs_Q^k=(-\bs_H^k,0,1)$, $\bs_H^k\in \partial H(\bx^k)$.
\end{lemma}
\begin{proof}
Using the definitions of  $\mW$, $q(\bw)$, $Q(\bw)$, we obtain that \eqref{subP:FW} is equivalent to 
\begin{align}
\begin{aligned}\label{subP:FW2}
\bw^{k+1} \in  \argmin_{\bw \in \mW}~& s_k-h(\bx^k)-\langle \bs_h^k, \bx-\bx^k\rangle+s-s^k + \frac{\beta}{2}\|\bm x - \bm x^k\|^2 \\
\rm s.t.\quad& t_k-H(\bx^k)-\langle \bs_H^k, \bx-\bx^k\rangle+t-t^k\le 0.
\end{aligned}
\end{align}
This is equivalent to \eqref{DC:subproblem} as there exists an optimal solution of \eqref{subP:FW2} satisfying $s=g(\bx)$ and $t=G(\bx)$.
\end{proof}

We next use the equivalent expression \eqref{P:FW} to give an equivalent characterization of KKT points (see Definition \ref{def:KKT}) of Problem \eqref{CCP:sample-DC} under the generalized MFCQ in Assumption \ref{AS:MFCQ}.
\begin{lemma}
Suppose that Assumptions \ref{AS:1} and \ref{AS:MFCQ} hold.  
Given $\bar{\bw} \in \mW$, $\bs_q \in \partial q(\bar{\bw})$ with $g(\bar\bx) \le \bar s,\ G(\bar\bx) \le\bar t$, and $\bs_Q \in \partial Q(\bar{\bw})$, if
\begin{align}\label{KKT2:P FW}
\langle \bs_q,  \bw - \bar{\bw}\rangle  +  \frac{\beta}{2}\|\bw-\bar\bw \|_T^2\ge 0
\end{align}
for all $\bw \in \mW$ satisfying $Q(\bar{\bw}) + \langle \bs_Q,\bw-\bar{\bw} \rangle \le 0$, then $\bar{\bx}$ is a KKT point of Problem \eqref{CCP:sample-DC}.
\end{lemma}
\begin{proof}
According to the statement of the lemma, we obtain that $\bar{\bw} \in \mW$ is an optimal solution to the following convex problem:
\begin{align*} 
\min_{\bw \in \mW}&\quad \langle \bs_q,\bw - \bar{\bw} \rangle +  \frac{\beta}{2}\|\bw-\bar\bw \|_T^2\\
\st &\quad Q(\bar{\bw}) + \langle \bs_Q,\bw-\bar{\bw} \rangle \le 0.
\end{align*}
Acccording to \Cref{lem:equi DC FW} with $\bw^k=\bar\bw$ and $\bx^k=\bar\bx$, the above problem is equivalent to
\begin{align*}
\min_{\bx \in \mX}\quad & g(\bm x) - h(\bar \bx)- \langle \bs_h, \bx-\bar{\bx} \rangle
+ \frac{\beta}{2}\|\bx-\bar\bx \|^2  \\
\st\quad &  G(\bm x) - H(\bar{\bx}) - \langle \bs_H, \bx-\bar{\bx} \rangle  \le 0,
\end{align*}
where $\bs_h \in \partial h(\bar{\bx})$ and $\bs_H \in \partial H(\bar{\bx})$. 

Therefore, we obtain that $\bar{\bx}$ is an optimal solution to the above convex problem.
This, together with the Slater's condition due to Assumption  \ref{AS:MFCQ}, implies that there exists $\lambda \in \R_+$ such that $(\bar{\bx},\lambda)$ satisfies 
\[
\begin{array}{ll}
\lambda\left(G(\bar\bx) - H(\bar\bx)\right) = 0,\ \bm{0} \in \partial g(\bar\bx) - \partial h(\bar\bx) + \lambda\left(\partial G(\bar\bx) - \partial H(\bar\bx)\right) + \mN_{\mX}(\bar\bx),
\end{array}
\]
which is just the KKT system of  Problem \eqref{CCP:sample-DC} in Definition \ref{def:KKT}. 
\end{proof}

Consequently, studying the iteration complexity of Algorithm \ref{alg-1} for computing an approximate KKT point of Problem \eqref{CCP:sample-DC} is equivalent to that of the variant of the FW iterations \eqref{subP:FW} for computing a point satisfying \eqref{KKT2:P FW}. 
However, we cannot expect to achieve a solution that satisfies \eqref{KKT2:P FW} in practice. Instead, we often obtain an approximate solution as shown in the next theorem, which can be seen as an approximation of a KKT point of Problem \eqref{CCP:sample-DC}.
The next theorem gives the iteration complexity for achieving an approximate solution.
\begin{thm}\label{thm:rate FW}
Suppose that Assumptions \ref{AS:1} and \ref{AS:MFCQ} hold.
Let $\{\bx^k\}$ be the sequence generated by Algorithm \ref{alg-1}. Then, there exists $\ell \in  \{1,\dots,k\}$  such that
\begin{align}\label{KKT2:P FWeps}
\langle \bs_q,  \bw -{\bw^\ell}\rangle  +  \frac{\beta}{2}\|\bw-\bw^\ell \|_T^2\ge -\frac{1}{k}\left(q(\bw^0) - q^*\right),
\end{align}
for all $\bw\in\mW$ and $Q(\bw^l)+\langle \bs^l_Q,\bw-\bw^l\rangle \le 0$, where $q^* \in \R$ is the optimal value of Problem \eqref{P:FW} and  $\bs^l_Q\in\partial Q(\bw^l)$.
\end{thm}
\begin{proof}
According to Lemma \ref{lem:equi DC FW}, a sequence $\{\bw^k\}$ generated by iterations \eqref{subP:FW} satisfies $\bw^k=(\bx^k,s^k,t^k)$ for all $k \ge 0$.
Since $q$ is a concave function and $\bs_q^k \in \partial q(\bw^k)$, we have
\begin{align*}
\langle \bs_q^k, \bw^{k} - \bw^{k+1} \rangle \le q(\bw^k) - q(\bw^{k+1}).
\end{align*}
Averaging the above inequality over $k$ yields
\begin{align*}
\frac{1}{k}\sum_{i=1}^k  \langle \bs_q^k, \bw^{k} - \bw^{k+1} \rangle \le \frac{1}{k}\left(q(\bw^0) - q(\bw^{k+1}) \right) \le \frac{1}{k}\left(q(\bw^0) - q^* \right),
\end{align*}
where the last inequality follows from the fact that $q^* \in \R$ is the optimal value of Problem \eqref{P:FW}. This implies that there exists an index $\ell \in \{1,\dots,k\}$ such that
\begin{align}\label{eq1:prop rate FW}
\langle \bs_q^\ell, \bw^{\ell} - \bw^{\ell+1} \rangle \le \frac{1}{k}\left(q(\bw^0) - q^* \right).
\end{align}
Moreover, it follows from the optimality $\bw^{k+1}$ to Problem \eqref{subP:FW} that for all $\bw \in \mW$ satisfying $Q(\bw^\ell) + \langle \bs_Q^\ell,\bw-\bw^\ell \rangle \le 0$,
\begin{align*}
\langle \bs_q^\ell,\bw^{\ell+1} - \bw^\ell \rangle +\frac{\beta}{2}\|\bw^{\ell+1} - \bw^\ell\|_T^2 \le \langle \bs_q^\ell,\bw - \bw^\ell \rangle +\frac{\beta}{2}\|\bw^\ell - \bw\|_T^2.
\end{align*}
This, together with \eqref{eq1:prop rate FW}, implies that it holds for all $\bw \in \mW$ satisfying $Q(\bw^\ell) + \langle \bs_Q^\ell,\bw-\bw^\ell \rangle \le 0$ that
\begin{align*}
\langle \bs_q,  \bw -{\bw^\ell}\rangle  +  \frac{\beta}{2}\|\bw-\bw^\ell \|_T^2 & \ge
\langle \bs_q^\ell,\bw^{\ell+1} - \bw^\ell \rangle +\frac{\beta}{2}\|\bw^{\ell+1} - \bw^\ell\|_T^2  \ge -\frac{1}{k}\left(q(\bw^0) - q^*\right).
\end{align*}
We complete the proof.
\end{proof}

We remark that in contrast to Theorems \ref{thm:subseq} and \ref{thm:entire conv} that require $\rho+2\beta >0$, Theorem \ref{thm:rate FW} can be applied to analyze the case of $\rho+2\beta \ge 0$. It is worth noting that when $\beta=0$, the standard iteration complexity of the FW method for general nonconvex problems is $O(1/\sqrt{k})$ (see, e.g., \cite{lacoste2016convergence}), but the iteration complexity of our proposed FW method is improved to $O(1/k)$ as we construct a concave minimization surrogate using the DC structure. 

\section{Extensions}\label{sec:exte}

In this section, we first discuss how to extend our approach to solve chance constrained problems with chance constraints estimated by general non-parametric estimation. We then extend the proximal DC algorithm for solving Problem \eqref{CCP:sample-DC} with multiple DC constraints, which can be used to solve chance constrained programs with multiple chance constraints.
\subsection{Extension to General Non-Parametric Estimation of the Empirical Quantile}\label{subsec:ext1}

We consider non-parametric estimators that can be represented as a linear combination of order statistics of a sample drawn from the population distribution. The main advantage of non-parametric estimators is that they are easy to calculate and often resistant to outliers. Due to this, non-parametric estimators have been widely used in the literature; see, e.g., \cite{cui2018portfolio,martins2018nonparametric}. 
This naturally motivates us to apply the non-parametric estimators to Problem \eqref{P:CCP sample}.

An L-estimator is a commonly used non-parametric estimator. Suppose that a sample of $N$ i.i.d. realizations $\{X_i\}_{i=1}^N$ of some unknown distribution $F_X$ is available. In general, L-estimators of the empirical quantile take the form $\sum_{i=1}^Nw_iX_{[i]}$,
where $\bm{w} \in \Delta := \left\{\bu \in \R^N: \b0 \le \bu \le \bo, \bo^T\bu = 1\right\}$. In statistics, there are many different L-estimators that outperform the empirical quantile in both theory and practice; see, e.g., \cite{dielman1994comparison,jadhav2009parametric,van2000asymptotic}. Then, we consider some typical L-estimators of the $p$ empirical quantile for $p \in (0,1)$, i.e., $X_{[M]}$, where $M=\lceil pN \rceil$. For instance, the weighted average at $X_{[M-1]}$ (see, e.g., \cite{dielman1994comparison,jadhav2009parametric}) defined as
\begin{align*}
L_1 = (1-g)X_{[M-1]} + gX_{[M]},
\end{align*}
where $g=Np-M+1$.

Another widely used non-parametric estimator is the kernel quantile estimator (see, e.g., \cite{li2007nonparametric,parzen1979nonparametric}) defined as
\begin{align*}
L_2 = \sum_{i=1}^N \left(  \int_{(i-1)/N}^{i/N}\frac{1}{h}K\left( \frac{x-p}{h} \right) dx  \right) X_{[i]},
\end{align*}
where $h > 0$ is a constant and $K(t)$ is a kernel function satisfying $\int_{-\infty}^\infty K(t)dt = 1$, $K(t)\ge 0$, and $K(-t) = K(t)$. It is worth noting that this kernel quantile estimator can be viewed as a smoothing version of the empirical quantile estimator.


We consider a more general form of non-parameter estimators $\sum_{i=1}^Nw_i\widehat{C}_{[i]}(\bx)$, where the weight $\bm{w} \ge 0$ is given. This covers L-estimators and kernel quantile estimators. Then we obtain the following surrogate of \eqref{P:CCP}:
\begin{align}\label{CCP:sample-L}
\min_{\bx \in \mX}\left\{f(\bx):\ \sum_{i=1}^Nw_i\widehat{C}_{[i]}(\bx) \le 0\right\},
\end{align}
It is worth pointing out that Problem \eqref{P:CCP sample} is actually a special case of Problem \eqref{CCP:sample-L} by taking $w_M=1$ and $w_i=0$ for all $i \neq M$. Then, we reformulate this problem into a DC constrained DC program. Before we proceed, let
\begin{align}\label{set:Z-L}
\bar{\mZ} := \left\{\bx\in \R^n:\ \sum_{i=1}^Nw_i\widehat{C}_{[i]}(\bx) \le 0\right\}.
\end{align}
Similar to Lemma \ref{lem:set-Z}, we can also express the above constraint as a DC constraint.
\begin{lemma}
Let
\begin{align}\label{eq:L G H}
G(\bx) := \sum_{i=1}^N w_i \sum_{j=i}^N \widehat{C}_{[j]}(\bx),\  H(\bx) :=  \sum_{i=1}^{N-1}w_i\sum_{j=i+1}^N \widehat{C}_{[j]}(\bx),
\end{align}
where $\bw \ge \bm 0$. Then, $G$ and $H$ are both continuous and convex functions, and the chance constraint in $\bar{\mZ}$ is equivalent to a DC constraint
\begin{align*}
  G(\bx) - H(\bx) \le 0.
\end{align*}
\end{lemma}
\begin{proof}
Using the argument in Lemma \ref{lem:set-Z}, we can show that $\sum_{j=i}^N \widehat{C}_{[j]}(\bx)$ for $i=1,\dots,N$ are convex functions. Since each of $G$ and $H$  in \eqref{eq:L G H} is a positive weighted sum of convex functions, $G$ and $H$ are both convex functions.
According to \eqref{eq:decomp zM}, we have for $i=1,\dots,N-1$,
\begin{align*}
\widehat{C}_{[i]}(\bx) =  \sum_{j=i}^N \widehat{C}_{[j]}(\bx) - \sum_{j=i+1}^N \widehat{C}_{[j]}(\bx).
\end{align*}
This yields that
\begin{align*}
\sum_{i=1}^Nw_i\widehat{C}_{[i]}(\bx) & = \sum_{i=1}^{N-1}w_i\widehat{C}_{[i]}(\bx) + w_N\widehat{C}_{[N]}(\bx)  = \sum_{i=1}^{N-1}w_i\left( \sum_{j=i}^N \widehat{C}_{[j]}(\bx) - \sum_{j=i+1}^N \widehat{C}_{[j]}(\bx) \right) + w_N\widehat{C}_{[N]}(\bx) \\
& = \sum_{i=1}^N w_i \sum_{j=i}^N \widehat{C}_{[j]}(\bx) -  \sum_{i=1}^{N-1}w_i\sum_{j=i+1}^N \widehat{C}_{[j]}(\bx) = G(\bx) - H(\bx).
\end{align*} 
\end{proof}

We then obtain a DC constrained DC program for L-estimators or kernel quantile esitmators  of the empirical quantile.
Consequently, we can still apply the proposed pDCA for solving the resulting problem.


\subsection{Extension to Multiple DC Constraints}\label{subsec:ext2}

In this subsection, we consider that Problem \eqref{CCP:sample-DC} has multiple DC constraints
\begin{align}\label{set:X1}
\ G_i(\bx) - H_i(\bx) \le 0,\ \text{for}\ i = 1,\dots,K,
\end{align}
where  $G_i:\R^n \rightarrow \R$ and $H_i:\R^n \rightarrow \R$ are continuous and convex functions.  That is, we consider the problem
\begin{align}\label{DC:mul}
\begin{aligned}
\min_{\bx \in \mX}\ \ f(\bx) := g(\bx) - h(\bx)\qquad \st\ \ G_i(\bx) - H_i(\bx) \le 0,\ \text{for}\ i = 1,\dots,K.
\end{aligned}
\end{align}
We can still apply the proximal DC algorithm for solving this problem. Specifically, suppose that an initial point $\bx^0 \in \mX$ satisfying $G_i(\bx^0) - H_i(\bx^0) \le 0,~i=1,\ldots,K$ is available. At the $k$-th iteration, we choose
$\bs_h^k \in \partial h(\bx^k)$ and $\bs_{H_i}^k \in \partial H_i(\bx^k)$ for $i = 1,\dots,K $, and generate the next iterate $\bx^{k+1}$ by solving the following convex subproblem
\begin{align}\label{DC:iterations}
\begin{aligned}
\bx^{k+1} \in \argmin_{\bx \in \mX} \quad & g(\bx) -  h(\bx^k) - \langle \bs_h^k, \bx-\bx^k \rangle  + \frac{\beta}{2}\|\bx-\bx^k\|^2 \\
\st\quad &  G_i(\bx) - H_i(\bx^k) - \langle \bs_{H_i}^k, \bx-\bx^k \rangle  \le 0,\ \text{for}\ i = 1,\dots,K,
\end{aligned}
\end{align}
where $\beta \ge 0$ is a penalty parameter. In particular, we can also prove subsequential convergence to a KKT point for the proximal DC algorithm by assuming the following generalized MFCQ:
\begin{assumption}[Generalized MFCQ] \label{AS:MFCQ-1}
The generalized MFCQ holds for Problem \eqref{DC:mul}, i.e., there exists $\by \in \mX$ such that
\begin{align*}
\langle \nabla \omega_i(\bx),\by-\bx\rangle < 0,\ \text{for all}\ i \in \mA(\bx),
\end{align*}
and if $G_i(\bm x) = H_i(\bm x)$, it holds that  
\begin{align*}
    G_i(\by) - H_i(\bx) - \inf_{\bs_{H_i} \in \partial H_i(\bx)} \langle \bs_{H_i}, \by - \bx \rangle < 0,\ i = 1,\dots,K.
\end{align*}
\end{assumption}
Using the similar argument in Section \ref{subsec:subseq}, we can obtain the following result:
\begin{coro}
Suppose that Assumptions \ref{AS:1} and \ref{AS:MFCQ-1} hold, the function $f$ is given in Problem \eqref{DC:mul}, $\mX$ is of the form of \eqref{set:X}, and the level set 
\[\left\{\bx \in {\mX}:\ f(\bx) \le f(\bx^0),~ G_i(\bx) - H_i(\bx) \le 0,\ \text{for}\ i = 1,\dots,K\right\}\]
is bounded. Let $\{\bx^k\}$ be the sequence generated by \eqref{DC:iterations}  with $\rho + 2\beta > 0$. Then, any accumulation point of $\{\bx^k\}$ is a KKT point of Problem \eqref{DC:mul}.
\end{coro}

\section{Experimental Results}\label{sec:expe}

In this section, we conduct experiments to study the performance of our proposed method on both synthetic and real data sets. For ease of reference, we denote our proposed method by pDCA (resp. DCA) when $\beta > 0$ (resp. $\beta=0$) in Algorithm \ref{alg-1}. A key step in implementing pDCA and DCA is to compute a subgradient of $H$ at an iterate $\bm x^k$. According to \Cref{lem:subg H}, we first need to compute an element in $\mathcal{M}_c^i(\bm x^k)$ (see \eqref{index:c i}) and $\mathcal{M}_H(\bm x^k)$ (see \eqref{index:h H}), respectively. Specifically, for the former one, we compute the function values of $c_j(\bm x^k,\hat{\bxi}^i)$ for all $j=1,\dots,m$ and obtain an element in the index set $\mM_c^i(\bm x^k)$ by finding an index $j^* \in \{1,\dots,m\}$ such that $c_{j^*}(\bm x^k,\hat{\bxi}^i)$ has the largest value. For the latter one, after we compute $C(\bm x^k,\hat{\bxi}^i)$ for all $i=1,\dots,N$ using \eqref{def:c}, we obtain an element in the index set $\mM_H(\bm x^k)$ by finding an index $(i_1^*,\dots,i_{N-M}^*) \in \mI$ such that  $\{C(\bm x^k,\hat{\bxi}^{i_t^*})\}_{t=1}^{N-M}$ is the $N-M$ largest elements in $\{C(\bm x^k,\hat{\bxi}^{i})\}_{i=1}^N$, where $\cal I$ is defined in \eqref{set:I}. Finally, using these and \Cref{lem:subg H}, we obtain a subgradient of $H$ at $\bm x^k$. 

We also compare our methods with some state-of-the-art methods, which are CVaR in \cite{nemirovski2007convex}, the bisection-based CVaR method (Bi-CVaR) in \cite[Section 4.1]{bai2021augmented}, which is a heuristic approach that combines binary search and CVaR and can improve the performance of CVaR, mixed-integer program (MIP) in \cite{ahmed2008solving}, an augmented Lagrangian decomposition method (ALDM) in \cite{bai2021augmented}, and a DC approximation-based successive convex approximation method (SCA) in \cite{hong2011sequential}.
In particular, we use the optimization solver \textsf{Gurobi} (version 9.5.2) for solving linear, quadratic, and mixed integer subproblems. All the experiments are conducted on a Linux server with 256GB  RAM and 24-core AMD EPYC 7402 2.8GHz CPU. Our codes are implemented in MATLAB 2022b and are available at \cite{wang2025proximal} and \url{https://github.com/INFORMSJoC/2024.0648}.
For pDCA, we update the penalty parameter $\beta$ in an adaptive manner. That is, we set $\beta^{k+1}=\beta^k/4$ for $k=0,1,2,\dots$
For pDCA on each data set, we explore three different settings of the regularization parameter $\beta^0$, i.e., we set $\beta^0=0.1, 1, 10$  for pDCA-1, pDCA-2 and pDCA-3, respectively.
We set the parameters of the remaining methods as those provided in the corresponding papers. 
For the tested methods DCA, pDCA, Bi-CVaR, ALDM, and SCA, we use the point returned by CVaR as their starting point. In each test, we terminate the tested methods when $|f^{k}-f^{k+1}|/\max\{1,|f^{k+1}|\} \le 10^{-6}$, for $k = 0,1,2,\dots$, 
or the running time reaches $1800$ seconds. Since we only check the running time at the end of each iteration, the actual finishing time of an algorithm may be longer than this limit. 

\subsection{VaR-Constrained Portfolio Selection Problem}

In this subsection, we study the VaR-constrained mean-variance portfolio selection problem, which aims to minimize the risk while pursuing a targeted level of returns with probability at least $1-\alpha$.  Let $\bm{\mu} \in \R^n$ and $\bm{\Sigma} \in \R^{n\times n}$ respectively denote expectation and covariance matrix of the returns of $n$ risky assets, and $\gamma \in \R_+$ denote the risk aversion factor. By letting $\bx \in \R_+^n$ denote the allocation vector such that the weight of the $i$-th risky asset is $x_i$ for $i \in [n]$, this problem is formulated as follows:
\begin{equation} \label{equ:psp}
\begin{aligned}
\min_{\bx \in \R^n} &\quad \gamma \bx^{T} \bm{\Sigma}\bx-\bm{\mu}^T\bx \qquad {\rm s.t.} \quad \mathbb{P}\left(\bm{\xi}^{T} \bm{x} \geq R\right) \geq 1-\alpha,\ \sum_{i=1}^n x_i = 1,\ 0 \leq x_i \leq u,\ i=1,\dots,n,
\end{aligned}
\vspace{-0.1cm}
\end{equation}
where $R \in \R_+$ is a prespecified level on the return and $u \in \R_+$ is an upper bound on the weights.

\vspace{-0.1in}
\begin{table}[H]
\footnotesize 
\caption{Comparison on the portfolio selection problem (averaged over 5 instances)}\label{table-1}
\begin{center}\vspace{-0.2in}
\begin{tabular}{c|| l| ccccccccc  ccccccr}
\hline
($\alpha$,$n$) &  & \textbf{ MIP}  & \textbf{ CVaR} & \textbf{ Bi-CVaR} & \textbf{ DCA} & \textbf{ pDCA-1} & \textbf{ pDCA-2}  & \textbf{pDCA-3} & \textbf{ ALDM}  & \textbf{ SCA} \\
\hline
 \multirow{3}{*}{$\begin{pmatrix}0.05\\100\end{pmatrix}$} &fval&-1.3550&-1.1861&-1.2592&-1.2860&-1.2897&-1.3037& \textbf{-1.3087}&-1.3221&-1.2732\\
&time&35.87 &0.1271 &1.868 &\textbf{0.4603}&0.7387&0.9553& 2.2919&3.576&0.8343\\
&prob&0.9500&0.9887&{0.9500}&0.9627&0.9587&0.9587&0.9540 &0.9420*&0.9593\\
\hline
\multirow{3}{*}{$\begin{pmatrix}0.05\\200\end{pmatrix}$} &fval& -1.3531&-1.1914&-1.2754&-1.2950&-1.2923&-1.3066& \textbf{-1.3169}&-1.3284&-1.2787\\
&time&1800 &0.3778 &5.013 &\textbf{1.683}&1.808&2.861&5.8706 &9.901&2.589\\
&prob&0.9500&0.9873&{0.9500}&0.9553&0.9560&0.9560& 0.9523&0.9447*&0.9580\\
\hline
\multirow{3}{*}{$\begin{pmatrix}0.05\\300\end{pmatrix}$} &fval& -1.3484 &-1.1830&-1.2629& -1.2935 &-1.2835&-1.2934&\textbf{-1.3040} &-1.3279&-1.2525\\
&time&1800&0.9473&12.26&7.403&\textbf{6.188}&8.749&12.9648 &19.59&6.890\\
&prob&0.9500&0.9853&{0.9500}&0.9529&0.9553&0.9553&0.9540 &0.9456*&0.9584\\
\hline
\multirow{3}{*}{$\begin{pmatrix}0.05\\400\end{pmatrix}$} &fval& -1.3719&-1.1939&-1.2886&-1.3143&-1.3206&\textbf{-1.3291}&-1.3266 &-1.3150&-1.2775\\
&time&1800 &1.861 &26.61 &20.07&\textbf{15.87}&16.46&26.0155 &24.01&16.26\\
&prob&0.9502&0.9860&{0.9500}&0.9547&0.9512&0.9512&0.9520 &0.9467*&0.9595\\
\hline\hline
\multirow{3}{*}{$\begin{pmatrix}0.1\\100\end{pmatrix}$} &fval& -1.4429&-1.2284&-1.3781&-1.3699&-1.3761&-1.3839&\textbf{-1.3913} &-1.3545&-1.3826\\
&time&7.376 &0.1262 &1.875 &0.7790&\textbf{0.7084}&0.9591&2.3541 &0.7826&0.8081\\
&prob&0.9000&0.9687&{0.9007}&0.9140&0.9113&0.9113&0.9080 &0.9093&0.9153\\
\hline
\multirow{3}{*}{$\begin{pmatrix}0.1\\200\end{pmatrix}$} &fval& -1.4244&-1.2371&-1.3815&-1.3772&-1.3764&\textbf{-1.3934}&-1.3912 &-1.3266&-1.3827\\
&time&1225 &0.3467 &5.093 &3.385&3.040&4.350&7.0798 &\textbf{0.3601}&3.582\\
&prob&0.9000&0.9620&{0.9007}&0.9087&0.9127&0.9127& 0.9053 &0.9193&0.9103\\
\hline
\multirow{3}{*}{$\begin{pmatrix}0.1\\300\end{pmatrix}$} &fval&-1.4410&-1.2284&-1.3999&-1.4015&-1.3959&\textbf{-1.4052}&-1.4014 &-1.3000&-1.3899\\
&time&1800 &0.9493 &12.32 &14.44&11.43&11.18&15.7715 &\textbf{0.8458}&11.16\\
&prob&0.9000&0.9633&{0.9000}&0.9053&0.9042&0.9042&0.9056 &0.9353&0.9107\\
\hline
\multirow{3}{*}{$\begin{pmatrix}0.1\\400\end{pmatrix}$} &fval&-1.4694&-1.2467&-1.4200&\textbf{-1.4352}&-1.4316&-1.4262& -1.4272&-1.3017&-1.4190\\
&time&1800 &1.833 &26.42 &31.05&32.69&27.70&40.6247 &\textbf{0.9201}&27.62\\
&prob&0.9000&0.9653&{0.9002}&0.9047&0.9067&0.9067&0.9055 &0.9412&0.9100\\
\hline
\end{tabular}
\end{center}
\end{table}

We use $2523$ daily return data of $435$ stocks included in Standard \& Poor's 500 Index between March 2006 and March 2016, which can be downloaded from \url{https://sem.tongji.edu.cn/semch_data/faculty_cv/xjz/ccop.html}. 
Following \cite{bai2021augmented}, we generate the data input by choosing $n=100,200,300,400$, respectively. For each $n$, we generate 5 instances from the daily return data set by randomly selecting $n$ stocks from the 435 stocks and $N=3n$ sample points $\hat{\bxi}^\ell$ for all $\ell \in [N]$ from the 2523 daily return data.
Then, we compute the sample mean $\bm{\mu}$ and sample covariance matrix $\bm{\Sigma}$ using these data.
We set the remaining parameters as follows: $R=0.02\%$,  $\gamma = 2$, and $u = 0.5$.
In \Cref{table-1} and the other two tables below for the other two experiments, we use ``fval" to denote the averaged returned objective value for the test problems, ``time" the averaged CPU time (in seconds), and ``prob" the empirical in-sample probability of the chance constraint, all of which are averaged over 5 instances. We highlight the best values except those of MIP and CVaR for items ``fval" and ``time" since MIP is not suitable for large-scale data sets and the solution returned by CVaR is too conservative. 

We observe from Table \ref{table-1} that although MIP achieves the lowest objective value, it is the most time-consuming. In addition, we observe that pDCA is slightly better than DCA and both pDCA and DCA generally outperform CVaR, Bi-CVaR, ALDM, and SCA in terms of the objective value. \Cref{table-1} also demonstrates that CVaR is the fastest method, while DCA and pDCA are comparable to the remaining ones. 
Finally, we also observe that the in-sample probabilities of DCA and pDCA are generally comparable to those of the other methods, except that ALDM fails to satisfy the chance constraint for $\alpha=0.05$ and sometimes is too conservative for $\alpha=0.1$.

\vspace{-0.1in}
\begin{table}[t]
\footnotesize
 \caption{Comparison on the probabilistic transportation problem  (averaged over 5 instances)}\vspace{-0.2in} \label{table-2}
\begin{center}
\begin{tabular}{c|| l| ccccccccc  ccccccr}
\hline
($\alpha$,$N$) &  & \textbf{ MIP}  & \textbf{ CVaR} & \textbf{ Bi-CVaR} & \textbf{ DCA} & \textbf{ pDCA-1} & \textbf{ pDCA-2} & \textbf{ pDCA-3}  & \textbf{ ALDM}  & \textbf{ SCA} \\
\hline
 \multirow{3}{*}{$\begin{pmatrix}0.05\\500\end{pmatrix}$}&	fval	&	4.2584	&	4.3843	&	4.3700	&	4.3262&\textbf{4.3239}		&	\textbf{4.3239}	&	4.3251	&	4.7091	&	4.1716	\\
&	time	&	73.89	&	1.796	&	22.84	&	\textbf{3.681}&427.5		&	405.2	&	503.1	&	58.76	&	6.697	\\
&	prob	&	0.9500	&	1.0000	&	0.9504	&	{0.9500}&0.9500		&	{0.9500}	&	{0.9500}	&	0.9504	&	0.8180*	\\
\hline
 \multirow{3}{*}{$\begin{pmatrix}0.05\\1000\end{pmatrix}$} &	fval	&	4.3655	&	4.5423	&	4.4931	&	4.4445& \textbf{4.4431}		&	4.4435	&	4.4467	&	4.8644	&	4.4447	\\
&	time	&	543.0	&	2.818	&	44.35	&	\textbf{5.895}&2064		&	2441	&	1915	&	50.63	&	73.90	\\
&	prob	&	0.9500	&	0.9984	&	{0.9500}	&	{0.9500}&0.9500		&	{0.9500}	&	{0.9500}	&	0.9636	&	0.9312*	\\
\hline
 \multirow{3}{*}{$\begin{pmatrix}0.05\\1500\end{pmatrix}$} &	fval	&	4.3946	&	4.6120	&	4.5067	&	\textbf{4.4631}&4.4647		&	4.4742	&	4.4891	&	4.8634	&	4.5818	\\
&	time	&	891.6	&	4.34	&	70.75	&	\textbf{12.66}&1928		&	1925	&	2002	&	44.63	&	261.5	\\
&	prob	&	0.9500	&	0.9980	&	0.9504	&	{0.9500}&0.9500		&	{0.9500}	&	{0.9500}	&	0.9787	&	0.9508	\\
\hline
 \multirow{3}{*}{$\begin{pmatrix}0.05\\2000\end{pmatrix}$}&	fval	&	4.4167	&	4.6538	&	4.5199	&	\textbf{4.4898}&4.4946		&	4.5063	&	4.5391	&	4.8597	&	4.5488	\\
&	time	&	1535	&	5.959	&	95.60	&	\textbf{14.99}&2298		&	2310	&	2447	&	46.52	&	336.7	\\
&	prob	&	0.9500	&	0.9848	&	0.9504	&	{0.9500}&0.9500		&	{0.9500}	&	{0.9500}	&	0.9843	&	0.9515	\\

\hline\hline
 \multirow{3}{*}{$\begin{pmatrix}0.1\\500\end{pmatrix}$}&	fval	&	4.1874	&	4.3833	&	4.3262	&	4.2591&4.2556		&	\textbf{4.2548}	&	\textbf{4.2548}	&	4.7110	&	4.3092	\\
&	time	&	171.6	&	1.626	&	24.75	&	\textbf{4.521}&570.2		&	528.5	&	591.7	&	42.70	&	65.16	\\
&	prob	&	0.9000	&	0.9916	&	{0.9000}	&	{0.9000}&0.9000		&	{0.9000}	&	{0.9000}	&	0.9812	&	0.9008	\\
\hline
 \multirow{3}{*}{$\begin{pmatrix}0.1\\1000\end{pmatrix}$} &	fval	&	4.2790	&	4.5306	&	4.3869	&	4.3617&4.3592		&	\textbf{4.3590}	&	4.3633	&	4.8027	&	4.4135	\\
&	time	&	674.5	&	2.928	&	47.76	&	\textbf{9.151}&1942		&	1944	&	1921	&	44.59	&	164.868	\\
&	prob	&	0.9000	&	0.9684	&	0.9002	&	{0.9000}&0.9000		&	{0.9000}	&	{0.9000}	&	0.9682	&	0.9028	\\
\hline
 \multirow{3}{*}{$\begin{pmatrix}0.1\\1500\end{pmatrix}$} &	fval	&	4.3031	&	4.5473	&	4.3975	&	\textbf{4.3694}&4.3696		&	4.3753	&	4.3937	&	4.7085	&	4.4092	\\
&	time	&	1673	&	5.073	&	74.30	&	\textbf{11.84}&1859		&	1899	&	1954	&	46.92	&	326.652	\\
&	prob	&	0.9000	&	0.9633	&	{0.9000}	&	{0.9000}&0.9000		&	{0.9000}	&	{0.9000}	&	0.9628	&	0.9041	\\
\hline
 \multirow{3}{*}{$\begin{pmatrix}0.1\\2000\end{pmatrix}$}&	fval	&	4.3212	&	4.5638	&	4.3998	&	\textbf{4.3805}&4.3866		&	4.4010	&	4.4280	&	4.7992	&	4.4406	\\
&	time	&	1801	&	5.982	&	102.8	&	\textbf{14.08}&2107		&	2217	&	2190	&	51.36	&	507.0	\\
&	prob	&	0.9000	&	0.9636	&	0.9001	&	{0.9000}&0.9000		&	{0.9000}	&	{0.9000}	&	0.9630	&	0.9110	\\
\hline
\end{tabular}
\end{center}
\vspace{0.1cm}
{\footnotesize The magnitude of fval is $10^{7}$.}
\vskip -0.2in
\end{table}

\subsection{Probabilistic Transportation Problem with Convex Objective}\label{subsec:exp ptp}

In this subsection, we consider a probabilistic version of the classical transportation problem, which has been widely studied in the literature; see, e.g., \cite{bai2021augmented,luedtke2010integer}. This problem is to minimize the transportation cost of delivering products from $n$ suppliers to $m$ customers. The customer demands are random and the $j$-th customer's demand is represented by a random variable $\xi_j$ for each $j \in \{1,\dots,m\}$. The $i$-th supplier has a limited production capacity $\theta_i \in \R_+$ for each $i \in \{1,\dots,n\}$. The cost of shipping a unit of product from supplier $i \in \{1,\dots,n\}$ to customer $j \in \{1,\dots,m\}$ is $c_{ij} \in \R_+$. Suppose that the shipment quantities are required to be determined before the customer demands are known. By letting $x_{ij}$ denote the amount of shipment delivered from supplier $i \in \{1,\dots,n\}$ to customer $j \in \{1,\dots,m\}$, this problem is formulated as
\begin{equation}\label{PTP}
\begin{aligned}
\min_{\bx \in \R^{n\times m}} \sum_{i=1}^n \sum_{j=1}^m c_{i j} x_{i j}\qquad \st &\quad \P\left( \sum_{i=1}^n x_{ij} \ge \xi_j,\ j = 1,\dots,m \right) \ge 1 - \alpha,  \\
&\quad \sum_{j=1}^m x_{ij} \le \theta_i,\ x_{ij} \ge 0,\ i =1,\dots,n,\ j = 1,\dots,m.
\end{aligned}
\end{equation}

In our experiments, we use the setting in \cite{luedtke2010integer} to generate parameters $(\bm{\theta},\bm{c},\hat{\bxi})$, which is downloaded from \url{http://homepages.cae.wisc.edu/~luedtkej/}. In particular, we choose $(n,m)=(40,100)$ and $N=500, 1000, 1500, 2000$. 
We report the experimental results in \Cref{table-2}. We observe that DCA and pDCA in general can find significantly better solutions than CVaR and ALDM, and  slightly better solutions than Bi-CVaR and SCA in terms of objective values. 
Meanwhile, we see that MIP returns either global optimal solutions or best objective values among all the algorithms in the time limit. 
We also observe that the CPU time of the DCA is less than Bi-CVaR and ALDM, much less than that of MIP and pDCA, and is slightly larger than that of CVaR.
We should mention that pDCA is the most time-consuming among the tested methods, since it solves a quadratic programming subproblem in each iteration, while other methods solve a linear programming subproblem. Table \ref{table-2} also indicates that the in-sample probabilities of DCA and pDCA are exactly the risk level $1-\alpha$ in all instances, while the in-sample probabilities of ALDM and SCA may be either too loose or too conservative.

\subsection{Probabilistic Transportation Problem with Non-Convex Objective}

In this subsection, we consider a probabilistic version of the classical transportation problem with a non-convex objective function, which has been studied in \cite{bai2021augmented,dentcheva2013regularization}.
This problem is to minimize the transportation cost of delivering products from $n$ suppliers to $m$ customers. The customer demands are random and the $j$-th customer's demand is represented by a random variable $\xi_j$ for each $j \in \{1,\dots,m\}$. The $i$-th supplier has a limited production capacity $\theta_i \in \R_+$ for each $i \in \{1,\dots,n\}$. The cost of shipping a unit of product from supplier $i \in \{1,\dots,n\}$ to customer $j \in \{1,\dots,m\}$ is $c_{ij} \in \R_+$. Suppose that the shipment quantities are required to be determined before the customer demands are known. Let $x_{ij}$ denote the amount of shipment delivered from supplier $i \in \{1,\dots,n\}$ to customer $j \in \{1,\dots,m\}$. 
Here, we assume that the transportation cost from supplier $i$ to customer $j$ consists of the normal cost $c_{ij}x_{ij}$ and cost discount $a_{i j} x_{i j}^2 \left(a_{i j}<0\right)$. Consequently, this problem can be formulated as
\begin{align}\label{eq:nonconvexPTP}
\begin{aligned}
\min_{\bx \in \R^{n\times m}} \quad \sum_{i=1}^n \sum_{j=1}^m  c_{i j} x_{i j}+a_{i j} x_{i j}^2\qquad \st & \quad \mathbb{P}\left(\sum_{i=1}^n x_{i j} \geq \xi_j, j=1, \ldots, m\right) \geq 1-\alpha, \\
&\quad \sum_{j=1}^m x_{i j} \leq \theta_i,\ x_{i j} \geq 0,\ i=1, \ldots, n,\ j=1, \ldots, m.
\end{aligned}
\end{align}
 In our experiments, we set $a_{i j}=-c_{i j} /\left(2 \theta_i\right)$ for all $i,j$, and the remaining setting is the same as that in the last section. Moreover, we use the setting in \cite{luedtke2010integer} to generate parameters $(\bm{\theta},\bm{c},\hat{\bxi})$, which is downloaded from \url{http://homepages.cae.wisc.edu/~luedtkej/}. In particular, we choose $(n,m)=(40,100)$ and $N=500, 1000, 1500, 2000$. 

 \vspace{-0.1in}
\begin{table}[!htbp]
 \caption{Comparison on the probabilistic transportation problem  (averaged over 5 instances)}  \label{table-3} 
 \begin{center}\small 
   \begin{tabular*}{\textwidth}{@{\extracolsep{\fill}}c|| l| cccccc  cccccr}
 \hline
 ($\alpha$,$N$) & & \textbf{ MIP} & \textbf{ DCA} & \textbf{ pDCA-1} & \textbf{ pDCA-2} & \textbf{ pDCA-3}  & \textbf{ ALDM}  & \textbf{ SCA}\\
  \hline
\multirow{3}{*}{$\begin{pmatrix}0.05\\500\end{pmatrix}$}&	fval	&	3.5098	&	3.6012	&3.5970	&	3.5973	&	\textbf{3.5962}	&	4.0023	&	3.4808	\\
&	time	&	1805	&	\textbf{7.448}	&281.3	&	340.7	&	458.8	&	267.6	&	8.42	\\
&	prob	&	0.9500	&	{0.9500}	&0.9500	&	{0.9500}	&	{0.9500}	&	0.9504	&	0.8180*	\\
 \hline
\multirow{3}{*}{$\begin{pmatrix}0.05\\1000\end{pmatrix}$}&	fval	&	3.5868	&	3.6830	&3.6871	&	\textbf{3.6822}	&	3.7027	&	4.1015	&	3.6819	\\
&	time	&	1803	&	\textbf{15.76}	&2006	&	1989	&	1851	&	178.6	&	87.53	\\
&	prob	&	0.9500	&	{0.9500}	&0.9500	&	{0.9500}	&	{0.9500}	&	0.9714	&	0.9318*	\\
\hline
\multirow{3}{*}{$\begin{pmatrix}0.05\\1500\end{pmatrix}$}&	fval	&	3.6123	&	\textbf{3.6888}	&3.7088	&	3.7170	&	3.7455	&	3.9974	&	3.7691	\\
&	time	&	1803	&	\textbf{23.12}	&2142	&	1927	&	1986	&	186.5	&	309.4	\\
&	prob	&	0.9500	&	{0.9500}	&0.9500	&	{0.9500}	&	{0.9500}	&	0.9845	&	0.9504	\\
\hline
\multirow{3}{*}{$\begin{pmatrix}0.05\\2000\end{pmatrix}$}&	fval	&	3.6237	&	\textbf{3.7133}	&3.7307	&	3.7575	&	3.7882	&	4.0842	&	3.7481	\\
&	time	&	1803	&	\textbf{33.04}	&2243	&	2381	&	2381	&	147.4	&	412.9	\\
&	prob	&	0.9500	&	{0.9500}	&0.9502	&	{0.9500}	&	{0.9500}	&	0.9845	&	0.9505	\\
\hline\hline
\multirow{3}{*}{$\begin{pmatrix}0.1\\500\end{pmatrix}$} &	fval	&	3.4581	&	3.5473	&3.5438	&	3.5436	&	\textbf{3.5421}	&	4.0195	&	3.5784	\\
&	time	&	1804	&	\textbf{8.845}	&335.0	&	413.1	&	405.3	&	175.4	&	67.68	\\
&	prob	&	0.9000	&	{0.9000}	&0.9000	&	{0.9000}	&	{0.9000}	&	0.9904	&	0.9016	\\
\hline
\multirow{3}{*}{$\begin{pmatrix}0.1\\1000\end{pmatrix}$} &	fval	&	3.5238	&	\textbf{3.6224}	&3.6272	&	3.6229	&	3.6406	&	3.9981	&	3.6503	\\
&	time	&	1802	&	\textbf{16.20}	&2065	&	1888	&	1949	&	151.1	&	201.2	\\
&	prob	&	0.9000	&	{0.9000}	&0.9000	&	{0.9000}	&	{0.9000}	&	0.9684	&	0.9010	\\
\hline
\multirow{3}{*}{$\begin{pmatrix}0.1\\1500\end{pmatrix}$} &	fval	&	3.5427	&	\textbf{3.6231}	&3.6422	&	3.6482	&	3.6779	&	4.0223	&	3.6499	\\
&	time	&	1802	&	\textbf{25.45}	&2043	&	1896	&	1976	&	177.2	&	401.5	\\
&	prob	&	0.9000	&	{0.9000}	&0.9004	&	{0.9000}	&	{0.9000}	&	0.9629	&	0.9007	\\
 \hline
\multirow{3}{*}{$\begin{pmatrix}0.1\\2000\end{pmatrix}$}&	fval	&	3.5521	&	\textbf{3.6281}	&3.6487	&	3.6775	&	3.7071	&	4.0006	&	3.6647	\\
&	time	&	1802	&	\textbf{27.14}	&2129	&	2242	&	2248	&	156.4	&	612.6	\\
&	prob	&	0.9000	&	{0.9000}	&0.9004	&	{0.9000}	&	0.9032	&	0.9631	&	0.9114	\\
\hline
 \end{tabular*}
 \end{center}
 \vspace{0.1cm}
{\footnotesize The magnitude of fval is $10^{7}$.}
\vskip -0.1in
 \end{table}   

Since the objective function of this problem is non-convex, CVaR and Bi-CVaR cannot handle this problem. Then, we only compare our proposed method with MIP, ALDM, and SCA. To generate a feasible initial point, we apply CVaR to solve Problem \eqref{eq:nonconvexPTP} without cost discount in the objective function. 
We report the experimental results in Table \ref{table-3}. We further point out that although MIP achieves the lowest objective value, it reaches the time limit for all the instances, which indicates the hardness of the additional non-convex term in the objective.  In terms of objective values and running time, we observe that DCA generally outperforms pDCA, ALDM, and SCA in most of cases. We should mention that pDCA is the most time-consuming among the tested methods except MIP, since it solves a quadratic programming subproblem in each iteration, while other methods solve a linear programming subproblem.
We observe that the in-sample probabilities of DCA and pDCA are generally closer to the risk level $1-\alpha$ than ALDM and SCA in all instances.

\vspace{-0.1in}
\begin{table}[t]
\small 
 \caption{Comparison on the norm optimization problem  (averaged over 5 instances)}  \label{table-4} 
 \begin{center}
   \begin{tabular*}{\textwidth}{@{\extracolsep{\fill}}c|| l| cccccccc}
 \hline
 ($\alpha$,$N$) & & {\bf MIP} & {\bf CVaR} & {\bf BiCVaR} & {\bf DCA} & {\bf pDCA1} & {\bf pDCA2} & {\bf pDCA3} & {\bf SCA}\\
  \hline
\multirow{3}{*}{$\begin{pmatrix}0.05\\500\end{pmatrix}$}& fval & -28.2120 & -26.8209 & -27.7280 & \textbf{-28.0164} & -27.9586 & -27.9656 & -27.9547 & -27.9810\\
& time & 1801 & 10.86 & 52.93 & 455.9 & 87.76 & 86.81 & 96.65 & \textbf{50.93}\\
& prob & 0.9500 & 0.9820 & 0.9516 & 0.9500 & 0.9516 & 0.9520 & 0.9508 & 0.9524\\
 \hline
\multirow{3}{*}{$\begin{pmatrix}0.05\\1000\end{pmatrix}$}& fval & -27.9205 & -26.5985 & -27.5185 & \textbf{-27.7318} & -27.6637 & -27.6774 & -27.6911 & -27.6879\\
& time & 1802 & 23.15 & \textbf{118.9} & 726.2 & 269.5 & 299.7 & 361.3 & 133.6\\
& prob & 0.9500 & 0.9808 & 0.9510 & 0.9506 & 0.9512 & 0.9504 & 0.9506 & 0.9532\\
\hline
\multirow{3}{*}{$\begin{pmatrix}0.05\\1500\end{pmatrix}$}& fval & -27.9565 & -26.6441 & -27.6711 & \textbf{-27.8045} & -27.7334 & -27.7561 & -27.7250 & -27.7481\\
& time & 1802 & 40.04 & 373.3 & 976.2 & 153.1 & 194.7 & \textbf{138.6} & 324.3\\
& prob & 0.9500 & 0.9808 & 0.9512 & 0.9504 & 0.9511 & 0.9504 & 0.9512 & 0.9543\\
\hline
\multirow{3}{*}{$\begin{pmatrix}0.05\\2000\end{pmatrix}$}& fval & -27.6895 & -26.4917 & -27.4546 & \textbf{-27.5653} & -27.5302 & -27.5224 & -27.5238 & -27.5572\\
& time & 1802 & 57.60 & 541.3 & 1342 & \textbf{266.0} & 272.7 & 292.1 & 280.1\\
& prob & 0.9503 & 0.9815 & 0.9511 & 0.9504 & 0.9504 & 0.9507 & 0.9507 & 0.9525\\
\hline\hline
\multirow{3}{*}{$\begin{pmatrix}0.1\\500\end{pmatrix}$} & fval & -28.9788 & -27.2620 & -28.4710 & \textbf{-28.6983} & -28.6256 & -28.6513 & -28.6140 & -28.6793\\
& time & 1802 & 9.369 & \textbf{45.10} & 594.8 & 88.44 & 103.3 & 86.35 & 59.73\\
& prob & 0.9000 & 0.9656 & 0.9024 & 0.9008 & 0.9008 & 0.9004 & 0.9004 & 0.9048\\
\hline
\multirow{3}{*}{$\begin{pmatrix}0.1\\1000\end{pmatrix}$} & fval & -28.8311 & -27.2811 & -28.5296 & -28.6618 & -28.6169 & -28.6390 & -28.6219 & \textbf{-28.6882}\\
& time & 1802 & 16.52 & \textbf{104.7} & 689.4 & 198.8 & 323.1 & 220.1 & 202.3\\
& prob & 0.9000 & 0.9602 & 0.9010 & 0.9006 & 0.9008 & 0.9002 & 0.9010 & 0.9034\\
\hline
\multirow{3}{*}{$\begin{pmatrix}0.1\\1500\end{pmatrix}$} & fval & -28.7416 & -27.2847 & -28.5288 & \textbf{-28.6682} & -28.6153 & -28.6104 & -28.6156 & -28.6276\\
& time & 1804 & 45.97 & 359.2 & 1353 & 211.2 & 170.6 & \textbf{169.2} & 285.5\\
& prob & 0.9000 & 0.9619 & 0.9015 & 0.9005 & 0.9012 & 0.9008 & 0.9008 & 0.9044\\
 \hline
\multirow{3}{*}{$\begin{pmatrix}0.1\\2000\end{pmatrix}$} & fval & -28.6728 & -27.2716 & -28.5037 & -28.5889 & -28.5594 & -28.5661 & \textbf{-28.6123} & -28.5627\\
& time & 1805 & 67.45 & 510.4 & 1341 & \textbf{288.7} & 323.8 & 496.8 & 298.3\\
& prob & 0.9010 & 0.9648 & 0.9044 & 0.9022 & 0.9037 & 0.9032 & 0.9002 & 0.9065\\
\hline
 \end{tabular*}
 \end{center}
\vskip -0.2in
 \end{table}

\subsection{Linear Optimization with Nonlinear Chance Constraint}

In this subsection, we consider an optimization problem with a linear objective and a joint convex nonlinear chance constraint, which has been studied in \cite{hong2011sequential,kannan2021stochastic}. Specifically, this problem takes the form
\begin{equation}
\min_{\bm x \in \R_+^d} \quad -\sum_{i=1}^{d}x_{i}\qquad \st \quad
\mathbb{P}\left(\sum_{i=1}^{d}\xi_{ij}^{2}x_{i}^{2}\leq \theta,\ j = 1,\dots,m \right) \geq 1-\alpha,
\end{equation}
where $\xi_{ij}$ for all $i,j$ are dependent normal random variables with mean $j/d$ and variance $1$, and $\mathrm{cov}(\xi_{ij},\xi_{i^\prime j}) = 0.5$ if $i \neq i^\prime$, $\mathrm{cov}(\xi_{ij},\xi_{i^\prime j^\prime}) = 0$ if $j \neq j^\prime$. In our experiments, we set $d = 20$,  $m = 20$, and $\theta = 100$. Moreover, we consider four different numbers of training samples, i.e., $N=500,1000,1500,2000$.

From Table 4, we observe that MIP has the lowest objective value in all cases, yet it is the most time-consuming. In all cases, pDCA or DCA achieves the lowest objective value except the case where $\alpha = 0.1, N = 1000$.  Moreover, it is worth mentioning that DCA and pDCA achieve in-sample probabilities close to the prespecified level, while the in-sample probabilities of CVaR and SCA tend to be more conservative.

\section{Conclusions}\label{sec:conc}

In this paper, we proposed a new DC reformulation based on the empirical quantile for solving data-driven chance constrained programs and proposed a proximal DC algorithm to solve it. We proved the subsequential and sequential convergence to a KKT point of the proposed method and derived the iteration complexity for computing an approximate KKT point.
We point out that our analysis holds for general DC constrained DC programs beyond those reformulated from chance constrained programs and can be extended to DC programs with multiple DC constraints.  
We also show possible extensions of our methods to nonparametric estimators for quantile in chance constrained programs. Finally, we demonstrated the efficiency and efficacy of the proposed method via numerical experiments. As future work, one interesting direction is to extend our analysis framework to the conic chance constraints \citep{van2023inner}.  

\section*{Acknowledgements}

We thank Dr. Lai Tian (The Chinese University of Hong Kong) for the fruitful discussion of the nonsmooth analysis of this work. We also thank Professor Ying Cui (University of California, Berkeley) for pointing out a technique error in Lemma 4 and bringing some references to our attention. 

\bibliographystyle{abbrvnat}
\bibliography{chance_constraint}

\begin{center}
{\Huge \bf Appendix}
\end{center} 
\smallskip 
\setcounter{section}{0}
\setcounter{lemma}{0}
\setcounter{defi}{0}
\setcounter{page}{1}
\pagenumbering{roman} 
\renewcommand\thesection{\Alph{section}}
\renewcommand{\thelemma}{\Alph{section}.\arabic{lemma}}
\renewcommand{\thedefi}{\Alph{section}.\arabic{defi}} 

\section{Auxiliary Definitions and Results on Convex Analysis}

We first present a lemma that provides a quadratic lower bound for strongly convex functions; see, e.g., \cite[Theorem 5.24]{beck2017first}.
\begin{lemma}\label{lem:sc_lb}
    If $f$ is a $\mu$-strongly convex function, then we have
    \[
    f(\by)\ge f(\bx) + \langle \bs, \by-\bx \rangle + \frac{\mu}{2}\|\by-\bx\|^2,\ \forall \bm s \in \partial f(\bm x).
    \]
\end{lemma}

We present some rules for calculating the subdifferential of the pointwise maximum of convex functions and the subdifferential of the sum of convex functions, as provided in \cite[Corollary E.4.3.2]{hiriart2004fundamentals} and \cite[Theorem 23.8]{rockafellar1970convex}, respectively. 

\begin{lemma}\label{lem:subdiff rule}
    Suppose that $f_1(\bm x),\dots,f_m(\bm x):\R^n \to \R$ are proper convex functions. \\
    (i) Let $f := \max\left\{f_1,\dots,f_m \right\}$.
    It holds that 
    \begin{align*}
        \partial f(\bm x) = \mathrm{conv}\left\{\cup \partial f_i(\bm x): i \in I(\bm x)\right\},
    \end{align*}
    where $I(\bm x) = \{i\in \{1,\ldots,n\}: f_i(\bm x) = f(\bm x)\}$ denotes the active index set at $\bm x$. \\
    (ii) Let $f = f_1+\dots+f_m$. If the convex sets $\mathrm{ri}(\mathrm{dom}(f_i))$ for all $i \in \{1,\ldots,m\}$ have a point in common, then 
    \begin{align*}
        \partial f(\bm x) = \partial f_1(\bm x) + \dots + \partial f_m(\bm x),\ \forall \bm x,
    \end{align*}
    where $\mathrm{ri}(C)$ denotes the relative interior of a convex set $C$. 
\end{lemma}



Then, we present a lemma that characterizes the normal cone of the convex set $\mathcal{X}$ defined in \eqref{set:X}. 

\begin{lemma}\label{lem:normal cone}
    Suppose that \Cref{AS:MFCQ} holds. Let $\mathcal{X} \subseteq \R^n$ be a convex set defined in \eqref{set:X}. It holds that 
    \begin{align*}
        \mN_{\mX}(\bx) = \left\{ \sum_{i\in \mE}u_i\ba_i + \sum_{i\in \mI}v_i\nabla \omega_i(\bx): \ v_i \ge 0,\ \text{for}\ i \in \mA(\bx),\ v_i = 0,\ \text{for}\ i \in \mathcal{I} \setminus \mathcal{A}(\bm x) \right\}.
    \end{align*}
\end{lemma}
\begin{proof}
 For ease of exposition, let 
\begin{align*}
    \mathcal{S} :=  \left\{ \sum_{i\in \mE}u_i\ba_i + \sum_{i\in \mI}v_i\nabla \omega_i(\bx): \ v_i \ge 0,\ \text{for}\ i \in \mA(\bx),\ v_i = 0,\ \text{for}\ i \in \mathcal{I} \setminus \mathcal{A}(\bm x) \right\}. 
\end{align*}
According to the definition of the normal cone for a convex set, we have for each $\bm x \in \mathcal{X} \subseteq \R^n$, 
\begin{align*}
    \mathcal{N}_{\cal X}(\bm x) = \left\{ \bm d \in \R^n: \langle \bm d, \bm y - \bm x \rangle \le 0,\ \forall \bm y \in \mathcal{X} \right\}. 
\end{align*}
For each $\bm d \in \mathcal{S}$, we compute for all $\bm y \in \mathcal{X}$, 
\begin{align*}
    \left\langle \sum_{i\in \mE}u_i\ba_i + \sum_{i\in \mI}v_i\nabla \omega_i(\bx), \bm y - \bm x \right\rangle = \sum_{i \in \mathcal{A}(\bm x)} v_i  \left\langle  \nabla \omega_i(\bx), \bm y - \bm x \right\rangle \le 0, 
\end{align*}
where the first equality follows from $\langle\bm a_i, \bm x \rangle = \langle \bm a_i ,\bm y \rangle = -b_i$ for all $i \in \mathcal{E}$ and $\bm v_i = 0$ for all $i \in \mathcal{I}\setminus \mathcal{A}(\bm x)$, and the inequality follows from $v_i \ge 0$ and $\left\langle  \nabla \omega_i(\bx), \bm y - \bm x \right\rangle  = \omega_i(\bm x) + \left\langle  \nabla \omega_i(\bx), \bm y - \bm x \right\rangle \le \omega(\bm y) \le 0 $ using $\omega_i(\bm x) = 0$ for all $i \in \mathcal{A}(\bm x)$ and the convexity of $\omega_i$. This implies $\bm d \in \mathcal{N}_{\cal X}(\bm x)$, and thus $\mathcal{S} \subseteq \mathcal{N}_{\cal X}(\bm x)$.

For ease of exposition, we write 
\begin{align*}
    \mathcal{X} = \left\{ \bm x \in \R^n: h_i(\bm x) = 0,\ \forall i \in \mathcal{E},\ \omega_i(\bm x) = 0,\ \forall i \in \mathcal{A}(\bm x),\ \omega_i(\bm x)  < 0,\ \forall i \in \mathcal{I} \setminus  \mathcal{A}(\bm x) \right\},
\end{align*}
where $h_i(\bm x) = \bm a_i^T\bm x + b_i$ for all $i \in \mathcal{E}$. According to \cite[Theorem 6.14]{RW04}, it holds that $\mathcal{N}_{\cal X}(\bm x)  \subseteq  \mathcal{S}$ at any $\bm x \in \mathcal{X}$ satisfying the following constraint qualification: the only vector $(\bm u, \bm v)$ satisfying $u_i \in \mathcal{N}_{\{0\}}(h_i(\bm x))$ for all $i \in \mathcal{E}$, $v_i \in \mathcal{N}_{\{0\}}(\omega_i(\bm x))$ for all $i \in \mathcal{A}(\bm x)$, and $v_i = 0$ (due to $v_i \in \mathcal{N}_{\{x: x < 0\}}(\omega_i(\bm x))$) for all $i \in \mathcal{I} \setminus  \mathcal{A}(\bm x)$ such that 
\begin{align}\label{eq1:lem cone}
    \sum_{i\in \mathcal{E}} u_i \bm a_i + \sum_{i \in \mathcal{A}(\bm x)} v_i \nabla \omega_i(\bm x)  = 0
\end{align}
is $u_i= 0$ for all $i \in \mathcal{E}$ and $v_i = 0$ for all $i \in \mathcal{A}(\bm x)$. Therefore, it remains to show the above constraint qualification. Without loss of generality, we assume that $\{\ba_i: i\in \mE\}$ is linearly independent, since otherwise, we can obtain the same results by eliminating the redundant linear equalities. According to \Cref{AS:MFCQ}, there exists $\bm y \in \mathcal{X}$ such that $\langle \nabla \omega_i(\bm x), \bm y - \bm x\rangle < 0$ for all $i \in \mathcal{A}(\bm x)$. Taking inner products with $\bm y - \bm x $ on both sides of \eqref{eq1:lem cone} yields 
\begin{align*}
    0 = \sum_{i \in \mathcal{A}(\bm x)} v_i \langle \nabla \omega_i(\bm x), \bm y - \bm x \rangle,
\end{align*}
where the equality follows from $\langle \bm a_i, \bm y - \bm x \rangle = 0$ for all $i \in \mathcal{E}$. Therefore, we have $v_i = 0$ for all $i \in \mathcal{A}(\bm x)$. Using this and linearly dependence of $\{\bm a_i\}_{i \in \mathcal{E}}$ yields $u_i=0$ for all $i \in \mathcal{E}$. Then, we proved the constraint qualification.  
\end{proof}

\section{Auxiliary Definitions and Results on Variational Analysis}

In this section, we introduce some definitions and lemmas from variational analysis that are used in our proofs. Specifically, we define the extended real domain $\overline{\R} = \R \cup \{+\infty\}$. First, we present the definition of the outer semi-continuity, as defined in \citet[Definition 5.4]{RW04}.

\begin{defi}[Outer Semi-Continuity]\label{def:osc}
    A set-value mapping $S:\R^n \rightrightarrows \R^m$ is outer semi-continuous at $\bar{\bm x}$ if 
    \begin{align*}
        \left\{ \bm u \in \R^n: \exists \bm x^k \to \bar{\bm x},\ \exists \bm u^k \to \bm u\ \text{with}\ \bm u^k \in S(\bm x^k)  \right\} \subseteq S(\bar{\bm x}). 
    \end{align*}
\end{defi}

In particular, the limiting subdifferential of proper function is outer semi-continuous, as shown in \citet[Proposition 8.7]{RW04}.

\begin{lemma}\label{lem:osc}
    For a function $f:\R^n \to \overline{\R}$ and a point $\bar{\bm x}$ where $f$ is finite, the mapping $\partial f$ is outer semi-continuous at  $\bar{\bm x}$. 
\end{lemma}

Next, we present the definition of strict continuity \cite[Definition 9.1]{RW04} and a sufficient condition to guarantee strict continuity. 

\begin{defi}[Strict Continuity]\label{def:cts}
    Let $f:\mathcal{D} \to \R$ be a function defined on a set $\mathcal{D} \subseteq \R^n$ and $\mathcal{S} \subseteq \mathcal{D}$. We say that $f$ is strictly continuous at $\bar{\bm x}$ relative to $\mathcal{S}$ if $\bar{\bm x} \in \mathcal{S}$ and the value 
    \begin{align*}
        \limsup_{\bm x, \bm x^\prime \underset{\mathcal{S}}{\to} \bar{\bm x}, \bm x \neq \bm x^\prime} \frac{|F(\bm x^\prime) - F(\bm x)|}{\|\bm x^\prime - \bm x\|}
    \end{align*}
    is finite. Then, $f$ is strictly continuous relative to $\mathcal{S}$ if for every point $\bar{\bm x} \in \mathcal{S}$, $f$ is strictly continuous at $\bar{\bm x}$ relative to $\mathcal{S}$. 
\end{defi}

\begin{lemma}\label{lem:cts}
   Consider the setting in \Cref{def:cts}. If a function $f$ is locally Lipschitz continuous on $\mathcal{S}$, then it is strictly continuous relative to $\mathcal{S}$. 
\end{lemma}





\cite[Corollary 10.50]{RW04} shows that one can characterize the Fr\'echet normal cone of a set via the extended chain rule. 

\begin{lemma}\label{lem:chain}
    Let $\mathcal{X} = F^{-1}(\mathcal{D})$ for a closed set $\mathcal{D} \subseteq \R^m $ and $F:\R^n \to \R^m$ be a strictly continuous mapping. At any $\bm x \in \mathcal{X}$, one has
    \begin{align*}
        \widehat{\mathcal{N}}_{\cal X}(\bm x) \supseteq \left\{\widehat{\partial} (yF)(\bm x): y \in \widehat{\mathcal{N}}_{\cal D}(F(\bm x)) \right\}. 
    \end{align*}
\end{lemma}

Finally, we present a lemma that provides some rules for calculating the subdifferential of functions. These rules directly follow from \cite[Theorem 8.6, Exercise 8.8(c), Proposition 10.5, Corollary 10.9]{RW04}. Notably, for a function $f:\R^n \to \overline{\R}$ and a point $\bm x$ with $f(\bm x)$ is finite, the {\em subderivative} $d f(\bm x): \R^n \to \overline{\R}$ is defined by 
\begin{align*}
    d f(\bm x)(\bm w) = \liminf_{\tau \searrow 0, \bm u \to \bm w } \frac{f(\bm x + \tau \bm u) - f(\bm x)}{\tau}. 
\end{align*}

\begin{lemma}\label{lem:rule sub}
    (i) For a function $f:\R^n \to \overline{\R}$ and a point $\bm x$ where $f$ is finite, then the subgradient sets $\partial f(\bm x)$ and $\widehat{\partial} f(\bm x)$ are closed with $\widehat{\partial} f(\bm x)$ being convex and $\widehat{\partial} f(\bm x) \subseteq \partial f(\bm x)$. 

    (ii) If $f = g + h$ with $g$ finite at $\bm x$ and $h$ smooth on a neighborhood of $\bm x$, then 
    \begin{align*}
        \widehat{\partial} f(\bar{\bm x}) = \widehat{\partial} g(\bar{\bm x}) + \nabla h(\bar{\bm x}). 
    \end{align*}

    (iii) Let $f(\bm x) = f_1(\bm x_1) + \dots + f_m(\bm x_m)$ for lower semicontinuous functions $f_i:\R^{n_i} \to \overline{\R}$, where $\bm x = (\bm x_1,\dots,\bm x_m)$ with $\bm x_i \in \R^{n_i}$. Then, at any $\bm x =  (\bm x_1,\dots,\bm x_m)$ with $f(\bm x)$ is finite and $d f_i(\bm x_i)(\bm 0) = \bm 0$, one has
    \begin{align*}
        \widehat{\partial} f(\bm x) = \widehat{\partial} f_1(\bm x_1) \times \cdots \times \widehat{\partial} f_m(\bm x_m).
    \end{align*}

    (iv) Let $f=f_1+\dots+f_m$ for proper and lower semicontinuous functions $f_i:\R^{n_i} \to \overline{\R}$ and $\bm x \in \mathrm{dom}(f)$. Then, we have 
    \begin{align*}
        \widehat{\partial} f(\bm x) \supseteq \widehat{\partial} f_1(\bm x) + \cdots + \widehat{\partial} f_m(\bm x). 
    \end{align*}
\end{lemma}

\section{Proof on Generalized MFCQ and its Equivalent Condition}\label{app:MFCQ}

Indeed, suppose that \Cref{AS:MFCQ} holds. Let $\bd=\by-\bx$. We immediately have \eqref{eq:mfcqeq1}. 
Using $G^\prime(\bm x, \bm d) = \inf_{t\ge 0} {(G(\bm x+t\bm d) - G(\bm x))}/{t}\le G(\bm x + d) - G(\bm x)$ and \eqref{slater:2}, we directly obtain \eqref{eq:mfcqeq2} when $t=1$. Conversely, suppose that \eqref{eq:mfcqeq1} and \eqref{eq:mfcqeq2} hold. Let $\bz=\bx+\alpha\bd$. For sufficiently small $\alpha>0$, we have
\[
\begin{aligned}
G(\bz) - H(\bm x) - \inf_{\bs_H\in\partial H(\bx)}\langle \bm s_H, \bz- \bm x \rangle & =  G(\bm x)  + \alpha G'(\bx, \bd) + o(\alpha)  - H(\bm x) - \alpha \inf_{\bs_H\in\partial H(\bx)}\langle \bm s_H,  \bm d \rangle \\
& = \alpha \left(G'(\bx,\bd) -\inf_{\bs_H\in\partial H(\bx)} \langle \bm s_H, \bd\rangle + o(1)\right) < 0. 
\end{aligned}
\]
where the first equality is due to the definition of the directional derivative and the second equality is due to $G(\bm x) = H(\bm x)$. Hence $\bz$ satisfies $\bz\in\mX$, \eqref{slater:1} and \eqref{slater:2}.

\section{Semialgebraic Functions and K\L\ Property}\label{app:D}

According to \cite[Section 5]{bolte2014proximal}, we provide some important definitions and results on the K\L\ property, as well as several concrete examples. 

\begin{defi}[Semialgebraic Sets and Functions]\label{def:semi}
    We say that a subset of $\R^n$ is semialgebraic if it can be written as a finite union of sets of the form 
    \begin{align*}
        \left\{\bm x \in \R^n: p_i(\bm x) = 0,\ q_i(\bm x) < 0, \forall i \right\},
    \end{align*}
    where $p_i$ and $q_i$ are real polynomial functions. Moreover, a function $f: \R^n \to \overline{\R}$ is semialgebraic if its graph is semialgebraic on $\R^{n+1}$.  
\end{defi}


There are a variety of sets and functions arising in optimization that are semi-algebraic.

\begin{exam}\label{exam:semi}
    The following sets and functions are semialgebraic:
    \begin{itemize}
        \item Real polynomial functions
        \item Indicator functions of semialgebraic sets 
        \item Finite sums and product of semialgebraic functions
        \item Sup/Inf type function, e.g., $\sup\{g(\bm x, \bm y): \bm y \in C\}$ is semialgebraic when $g$ is a semialgebraic function and $C$ is a semialgebraic set.
    \end{itemize}
\end{exam} 

\end{document}